\documentclass[11pt,reqno, english]{amsart}

\usepackage{amsfonts}
\usepackage{amssymb}
\usepackage{amsthm}
\usepackage{amsmath}
\usepackage[top=4cm, bottom=3.5cm, left=2cm, right=2cm,marginparwidth=1.8cm, marginparsep=0.1cm]{geometry}

\newtheorem{proposition}{Proposition}[section]
\newtheorem{theorem}[proposition]{Theorem}
\newtheorem{lemma}[proposition]{Lemma}
\newtheorem{corollary}[proposition]{Corollary}

\theoremstyle{definition}
\newtheorem{definition}[proposition]{Definition}

\theoremstyle{remark}
\newtheorem{remark}[proposition]{Remark}

\newtheorem*{remark*}{Remark}
\newtheorem*{remarks*}{Remarks}

\numberwithin{equation}{section}

\newcommand\B{\mathbb{B}}
\newcommand\Z{\mathbb{Z}}
\newcommand\R{\mathbb{R}}

\newcommand\C{\mathbb{C}}
\newcommand\N{\mathbb{N}}
\newcommand\G{\mathbb{G}}
\newcommand\HH{\mathbb{H}}
\newcommand\SSS{\mathbb{S}}
\newcommand\A{\mathbb{A}}

\newcommand\GL{\operatorname{GL}}

\newcommand\PSL{\operatorname{PSL}}

\newcommand\F{\mathbb{F}}
\newcommand\Q{\mathbb{Q}}

\def\g{\mathfrak{g}}

\def\a{\mathfrak{a}}

\newcommand\eps{\varepsilon}

\newcommand{\diam}{\operatorname{diam}}

\newsavebox{\proofbox}
\savebox{\proofbox}{\begin{picture}(7,7)  \put(0,0){\framebox(7,7){}}\end{picture}}

\begin{document}

\title{Lectures on approximate groups and Hilbert's 5th problem}
\author{Emmanuel Breuillard}
\address{Laboratoire de Math\'ematiques\\
B\^atiment 425, Universit\'e Paris Sud 11\\
91405 Orsay\\
FRANCE}
\email{emmanuel.breuillard@math.u-psud.fr}

\begin{abstract} This paper gathers four lectures, based on a mini-course at IMA in 2014, whose aim was to discuss the structure of approximate subgroups of an arbitrary group, following the works of Hrushovski and of Green, Tao and the author. Along the way we discuss the proof of the Gleason-Yamabe theorem on Hilbert's 5th problem about the structure of locally compact groups and explain its relevance to approximate groups. We also present several applications, in particular to uniform diameter bounds for finite groups and to the determination of scaling limits of vertex transitive graphs with large diameter.
\end{abstract}

\maketitle

\setcounter{tocdepth}{1}

\tableofcontents

\section{Additive combinatorics on groups}

\subsection{Introduction}
A key idea of Geometric Group Theory as initiated by Gromov and others in the seventies and eighties is to study infinite groups as geometric objects via their Cayley graphs. Endowed with the graph distance, the Cayley graphs are naturally metric spaces. Typically we look at a group $\Gamma$ generated by a finite symmetric set $S$ and define the associated Cayley graph and word metric:

 $$d_S(g,h):= \inf \{n \geq 0, g^{-1}h \in S^n \},$$
 where $S^n:=S \cdot \ldots \cdot S$ is the product set of $n$ copies of $S$. In particular the ball of radius $n$ around the identity $B_S(e,n)$ is nothing other than $S^n$. Usually people are interested in properties that are independent of the choice of $S$, and look at very large values of $n$. So $S$ is of bounded size and $n$ is large.

Additive Combinatorics on groups on the other hand is concerned with product sets $A^n:= A \cdot \ldots \cdot A$, where $n$ is of bounded size, but $A$ is a very large finite subset of an ambient group $G$. Traditionally, combinatorialists and number theorists have been interested in studying large finite subsets of integers that do not grow much under addition. This is an old and well-developed field of mathematics (much older than Geometric Group Theory) with applications to many number theoretical questions (such as unique factorization in rings of algebraic integers, arithmetic progressions in primes, Goldbach's problem, Waring's problem, etc).

Only fairly recently did people start looking at similar questions on non-abelian groups, only to find out that many of the tools developed in the commutative context could also be applied in a non-commutative setting.


\subsection{Examples of results from additive combinatorics}

To fix ideas, let us recall a sample of landmark results from additive combinatorics:

\begin{theorem}[Sum-product theorem, Bourgain-Katz-Tao \cite{bkt}, Bourgain-Glibichuk-Konyagin \cite{bgk}] Given $\eps>0$, there is $\delta>0$ such if $p$ is a prime and $A \subset \F_p$ is a subset of the field $\F_p$ with $p$ elements, then either $|A+A|> |A|^{1+\delta}$ or $|AA|> |A|^{1+\delta}$ or $|A|> p^{1-\eps}$.
\end{theorem}

In other words a subset of a prime field must grow either by addition or by multiplication unless it is almost all of the field. This can be viewed as a quantitative version of the fact that there are no non-trivial subrings in $\F_p$. Of course it is easy to construct large subsets $A$ of $\F_p$ that do not grow much under addition, or under multiplication: take a finite arithmetic progression, or geometric progression; then $|AA|\leq 2|A|$, but $|A|$ can be of arbitrary cardinality inside $\F_p$.

By now pretty good quantitative estimates have been obtained on the exponent $\delta(\eps)$, see e.g. \cite{garaev1, garaev2, rudnev} and references therein.

\begin{theorem}[Freiman-Ruzsa-Chang theorem \cite{freiman1, ruzsa, chang}]\label{freiman} Let $A \subset \Z$ be a finite subset such that $|A+A|\leq K|A|$. Then $A$ is $C(K)$-commensurable to a $d$-dimensional arithmetic progression $P$ with $d=O(\log K)$.
\end{theorem}


Here we used the following terminology:

\begin{definition}[K-commensurability] Two finite subsets $A$ and $B$ of an ambient group $G$ are called $K$-commensurable if
$$|A \cap B| \geq \frac{1}{K} \max\{|A|,|B|\}.$$
\end{definition}

and

\begin{definition}[Multidimensional arithmetic progression] A $d$-dimensional arithmetic progression in an ambient group $G$ is a finite subset of the form $P:=\pi(B)$, where $\pi: \Z^d \to G$ is a group homomorphism and $B:=\prod_{i=1}^d [-L_i,L_i]$ is a ``box'' with side-lengths $L_i \in \N$.
\end{definition}

\bigskip

\noindent {\bf Remarks:}

a) It is easy to see that $P$ as above satisfies $|PP| \leq 2^d |P|$ because $B+B$ is contained in at most $2^d$ translates of $B$.

b) The proof of the Freiman-Ruzsa-Chang theorem is based on harmonic analysis on $\Z/p\Z$ for some suitably chosen prime $p$. We refer the reader to \cite{bilu}, \cite{tao-vu} and \cite{granville} for nice accounts with proofs. The constant $C(K)$ given by these authors was historically of order $e^{O(K^{O(1)})}$. Tom Sanders recently gave much better, almost polynomial, bounds. See \cite{sanders-survey} for a recent survey. Conjecturally (``Polynomial Freiman-Ruzsa conjecture'', see \cite{green-polynomial}) one ought to be able to take $C(K)$ to be of order $O(K^{O(1)})$.

In this theorem and its applications, the quality of the bounds (i.e. their dependence on $K$) are important.\\

\bigskip

Non-commutative analogues of the above results have started to draw the attention of a number of people in the past ten years. For example, the sum-product theorem can be reformulated as a Freiman type theorem in the affine group of the line (see \cite{helfgott-sl3}). The main goal of these lectures is to present a generalization of Freiman's theorem to arbitrary groups, following \cite{hrushovski} and \cite{bgt}. See \cite{bgt-survey} for another survey on this topic. Proofs were completely avoided in \cite{bgt-survey}.  Our intent in these lectures is not to provide a complete treatment, but rather to give some ideas regarding the proofs, a rather complete outline and even the details of some steps, while remaining as informal as possible. All details can be found in the original paper \cite{bgt} and full proofs are also given in Tao's recent book \cite{tao-hilbert}. For yet another account, we recommend Van den Dries' Bourbaki report \cite{DriesBourbaki}. Also part of our focus here is also  on the geometric applications, especially in the last lecture, where the main ones are given full proofs.

\bigskip

We refer the reader to \cite{Kemp, Frei, BF, HLS, Zem, Ham} for some older results about the general non-commutative case, mostly dealing with the small $K$ regime. In the arbitrary $K$ regime in the non-commutative context, early papers include \cite{elekes-kiraly,elon,tao, helfgott-sl2}. One of the first results in this direction is due to Razborov \cite{razborov} and deals with finite subsets of the free group.  Razborov proved that if $F$ is a non-abelian free group and $A \subset F$ a finite subset, then $$|AAA| \geq |A|^2/(\log |A|)^c$$ for some positive $c$ independent of $A$ and $F$, unless $A$ lies in a cyclic subgroup. The exponent $2$ is best possible here, as is clear by considering for example $A:=\{x\}\cup \{y^i ; -N<i<N\}$ in $F_2:=\langle x,y\rangle$.

S. R. Safin recently improved this and showed:

\begin{theorem}[Safin \cite{safin}] There is $c\in (0,1)$ such that if $F$ is a free group, then for every $n \geq 3$ and every finite subset $A$ in $F$
$$|A^n| \geq c^n|A|^{[\frac{n+1}{2}]}$$
unless $A$ generates a cyclic group.
\end{theorem}

Safin's proof gives an explicit lower bound on $c$. One can take $c=\frac{1}{62}$. Again the same example shows that the exponent is sharp for each $n$. See \cite{button} for further generalizations of the Razborov-Safin results to free products and HNN extensions.
\bigskip

\noindent {\bf Remark:} When $G$ is an abelian group, then $|AA| \leq K|A|$ implies that $|A^n| \leq O_{K,n}(1) |A|$ for every $n\geq 2$, according to the Pl\"unnecke-Ruzsa inequalities (see \cite{tao-vu}). This is still the case in the non-abelian context, but when starting with the assumption $|AAA| \leq K|A|$ (see \S \ref{tool} below). A set $A$ consisting of a subgroup $H$ and another element $x$ with $H \cap xHx^{-1}=\{1\}$ yields an example with small $|AA|$ and large $|AAA|$. This justifies the consideration of three-fold product sets in non-commutative groups. See \S \ref{tool} below for more on this issue.

\subsection{The Freiman inverse problem}
The general problem (``Freiman inverse problem'' as advertised in particular by T. Tao in his blog several years ago) is to describe the structure of large subsets $A$ of an arbitrary ambient group $G$ that satisfy $|AA|\leq K|A|$ for some fixed constant $K$.

The value $|AA|/|A|$ is usually called the \emph{doubling constant} of $A$, equivalently we say that $A$ has \emph{doubling at most $K$}.

For example we easily have:

\begin{proposition}\label{extcase} Let $G$ be an arbitrary group. Suppose $|AA|=|A|$. Then there is a finite subgroup $H$ of $G$ and $A=aH=Ha$ for every $a \in A$.
\end{proposition}

\begin{proof}Pick $a \in A$ and set $H:=Aa^{-1}$. The set $H$ contains the identity, so $A \subset AH$. But $|AH|=|AHa|=|AA|=|A|$, hence $A=AH$. Iterating $A=AH^n$ for every $n\geq 1$ and it follows that the semigroup $\cup_n H^n$ is finite and contained in $a^{-1}A$. But a finite semigroup inside a group is a finite group. So the subgroup generated by $H$ satisfies $\langle H \rangle \subset a^{-1}A$. In particular $|\langle H \rangle| \leq |A|=|H|$. Hence $\langle H \rangle=H$. It follows that $A=Ha$ and that $H=AA^{-1}$. Finally $|H|=|A|=|AA|=|HaHa|=|HaH|=|a^{-1}HaH|$. But $a^{-1}HaH/H \simeq a^{-1}Ha/(a^{-1}Ha \cap H)$. Hence $a^{-1}Ha = H$ and we are done.
\end{proof}

Slightly trickier is the following:

\begin{proposition}\label{freimanlem} Let $G$ be an arbitrary group. Suppose $|AA|\leq 1.1|A|$. Then there is a finite subgroup $H$ of $G$ and $a \in A$ such that $aHa^{-1}=H$, $|H|\leq 1.2|A|$ and $A \subset aH$.
\end{proposition}

\begin{proof} Let $1_A$ be the indicator function of $A$ and $\|f\|_1 = \sum_{g \in G} |f(g)|$ the $\ell^1$-norm. We will prove that $H:=AA^{-1}$ is stable under multiplication using the following observation: if $||1_A - 1_{gA}||_1 < 2|A|$, then $g \in H$.
First we show that $H=A^{-1}A$ and that $|H|\leq 1.2|A|$.

 Note that $\forall x,y \in G$, $||1_{xA} - 1_{yA}||_1 = |xA \Delta yA| = 2(|A| - |xA \cap yA|)$. But if $x,y \in A$, then $|xA \cup y A|= 2|A|- |xA \cap y A| \leq |A^2|\leq 1.1|A|$. Hence  $\forall x,y \in A$,  $|xA \cap y A| \geq 0.9|A|$ and
\begin{equation}\label{first}
||1_{A} - 1_{x^{-1}yA}||_1  \leq 0.2|A|.
\end{equation}
Since $|xA \cap y A|>0$ we conclude that $xa=yb$ for some $a,b \in A$, i.e. $x^{-1}y \in AA^{-1}$.  Hence $A^{-1}A=H=AA^{-1}$. In fact $xa=yb$ for at least $0.9|A|$ pairs $(a,b)$ in $A \times A$. Therefore $|H|=|A^{-1}A|\leq |A|^2/0.9|A| \leq 1.2|A|$.

To see that $HH \subset H$ we write:
$$||1_{A} - 1_{z^{-1}wx^{-1}yA}||_1 \leq ||1_{A} - 1_{z^{-1}wA}||_1  + ||1_{z^{-1}wA} - 1_{z^{-1}wx^{-1}yA}||_1,$$ we get from $(\ref{first})$
$$||1_{A} - 1_{z^{-1}wx^{-1}yAA}||_1  \leq 0.4|A|$$
for every $x,y,z,w \in A$. Hence $HH \subset H$. Hence $H$ is a subgroup and $A \subset Ha$ for all $a \in A$, so $H=A^{-1}A \subset a^{-1}Ha$ and $a$ normalizes $H$.
\end{proof}

\begin{remark} The above  proof can be pushed with little effort to yield that if $|AA|< \frac{3}{2}|A|$, then $H=AA^{-1}=A^{-1}A$ is a finite subgroup of size $<\frac{3}{2}|A|$ which is normalized by $A$. This is a theorem of G. Freiman, see \cite{freiman2}. Clearly $3/2$ is sharp, take $A=\{0,1\} \in \Z$.\\
\end{remark}

\begin{remark}\label{diambound} An immediate consequence of the above lemma is the following general fact: if $S$ is a symmetric generating subset of a finite group $G$ with $|S| \geq \alpha |G|$ for some $\alpha \in (0,1)$, then $S^n=G$ for some $n \leq O(\log(\frac{1}{\alpha}))$. Indeed $|S^{2^n}| > 1.1^n|S|$ unless $|S^{2^{n-1}}| \geq |G|/1.2 > |G|/2$, in which case $S^{2^n}=G$.
\end{remark}

\begin{remark} One may ask: for which values of $K \geq 1$ is it true that the condition $|AA|\leq K|A|$ implies that $A$ is contained in boundedly many (in terms of $K$) cosets of a finite subgroup ? The answer is : this holds for all $K<2$. This is due to Yayha Hamidoune. He showed that $A$ is then contained in at most $\frac{1}{2-K}$ cosets of a subgroup of size at most $|A|$. See  \cite[Prop. 4.1]{hamidoune}. Clearly it is sharp because the arithmetic progressions $A:=\{0,1,\ldots,N\} \subset \Z$ have $|A+A|\leq 2|A|$.
\end{remark}

What if we increase the parameter $K$ beyond the $K=2$ threshold ? The next natural question to ask is: for which values of $K$ is it true that the condition $|AA| \leq K|A|$ implies that $A$ is contained in an extension $HL$, where $H$ is a finite subgroup normalized by $L$ a $d$-dimensional arithmetic progression ? Clearly such a set has doubling at most $2^d$.

However this question is ill-posed, because there are sets of doubling say $2^4$ which are not of this form: take a ball $S^n$ of large radius $n$ in the discrete Heisenberg group\\

$$H_3(\Z):= \big\{ \left(
               \begin{array}{ccc}
                 1 & x & z \\
                 0 & 1 & y \\
                 0 & 0 & 1 \\
               \end{array}
             \right)
             ; x,y,z \in \Z \big\}$$\\

Here $4$ is the homogeneous asymptotic dimension of the Heisenberg group, namely $|S^n| \simeq n^4$ up to multiplicative constants, for any given finite generating set $S$ of $H_3(\Z)$.\\

In 2008 Elon Lindenstrauss proposed \cite{elonpers} a general conjectural statement, which was subsequently proved by Ben Green and Terry Tao and the author in \cite{bgt}, and is the main subject of these notes:

\begin{theorem}[Structure theorem, weak version]\label{BGT} Let $G$ be a group, $K \geq 1$ a parameter and $A$ a finite subset with $|AA|\leq K|A|$. Then there is a coset of a virtually nilpotent subgroup of $G$ which intersects $A$ in a subset of size $\geq |A|/C(K)$.
\end{theorem}

\bigskip

\noindent {\bf Remarks.}

a) Here $C(K)$ is a number depending on $K$ only, and not on $A$ nor $G$.\\

b) The proof also gives a uniform bound on the nilpotency class and the number of generators of the nilpotent subgroup (both in $O(\log K)$ with an absolute implied constant), although no bound can be expected on the index (e.g. take $A$ a large finite simple group). More precisely we can take the virtually nilpotent group to be a finite-by-nilpotent group, where the nilpotent quotient satisfies $\sum_i d(C^{i}(N)/C^{(i+1)}(N))=O(\log K)$, where $(C^i(N))_i$ is the lower central series and $d(\cdot)$ the number of generators.\\

c) However the proof gives no explicit bound on $C(K)$. The proof is not effective.\\

d) Stated as above the theorem says nothing about the case when $G$ is finite. And this is enough for a number of geometric applications. However the proof does. We actually show that the virtually nilpotent subgroup can be taken to be finite-by-nilpotent where the finite normal subgroup lies in $XA$ for some $X$ of size $|X|\leq c(K)$ for some constant $c(K)$ depending only on $K$. We will give a stronger form of the theorem below, which includes this refinement.\\

e) The theorem had been established earlier by Tao \cite{tao-solvable} in the special case when $G$ is assumed solvable of bounded derived length. \\

f) In certain situations, most notably when $G$ is a free group (by the result of Razborov mentioned above) or when $G=GL_n(k)$ for some field $k$, then pretty good explicit bounds can be found on $C(K)$. For example in \cite{bgt1} it is shown that $C(K)=O(K^{O_n(1)})$ for $G=GL_n(\C)$. See also \cite{pyber-szabo, gill-helfgott} for the positive characteristic situation.\\


Theorem \ref{BGT} can be seen as a finitary version of Gromov's theorem on groups with polynomial growth. And indeed, as we now show, it easily implies Gromov's theorem.

\begin{theorem}[Gromov's polynomial growth theorem]\label{gromov} Let $\Gamma:=\langle S \rangle$ be a finitely generated group with polynomial growth (i.e. $|S^n| =O(n^D)$ for some $D>0$ as $n \to +\infty$). Then $\Gamma$ is virtually nilpotent.
\end{theorem}

\begin{proof}[Proof that Theorem \ref{BGT} implies Theorem \ref{gromov}] Since $|S^n|=O(n^D)$, there must be arbitrarily large $n$'s for which $|S^{4n}| \leq 5^D |S^n|$. Indeed, suppose not, then $|S^{4n}|>5^D |S^n|$ for all $n$ large enough, $n \geq n_0$ say, and hence $5^{Dk} |S^{n_0}| < |S^{4^kn_0}| =O(4^{Dk}n_0^D)$ a contradiction for large $k$.

Pick such an $n$. By Theorem \ref{BGT} applied to $A:=S^n$ and $K=5^D$, there is a virtually nilpotent subgroup $H_n$ such that $S^n$ has a large intersection with a coset of $H_n$. In particular $|S^{2n} \cap H_n| \geq |S^n|/C(D)\geq |S^{4n}|/(5^DC(D))$. Theorem \ref{gromov} follows by applying the following lemma with $k=2n$.

\begin{lemma} \label{schreier} In a group $\Gamma=\langle S\rangle$ if $|S^k \cap H| > |S^{2k}|/C$ for some subgroup $H$, some $C>0$ and some $k\geq C$, then $[\Gamma : H] \leq C$.
\end{lemma}

\begin{proof} Look at the Schreier graph of $\Gamma/H$. It is connected, so either $S^kH=\Gamma$ or $|S^kH/H| > k$. Write $S^k= \cup_{i=1}^N S^k \cap g_iH$. Clearly $\cup_{i=1}^N (S^k \cap g_iH)(S^k \cap H) \subset S^{2k}$. Hence $N|(S^k \cap H)|\leq |S^{2k}| < C |S^k \cap H|$. And hence $N=|S^kH/H| < C \leq k$. Hence $S^kH=\Gamma$, and $[\Gamma : H]=N < C$.
\end{proof}
\end{proof}

\begin{remark} In fact the above proves slightly more than Gromov's theorem. It shows that given $C,D>0$ there is $n_0=n_0(C,D)$ such that if $|S^n| \leq C n^D$ for some $n \geq n_0$, then $\Gamma$ is virtually nilpotent. Note that this is stronger than Gromov's own finitary version of his result (discussed at the end of his original paper \cite{gromov}). He showed (basically just by contradiction) that his theorem implies that there is $n_0=n_0(C,D)$ such that if $|S^n| \leq C n^D$ holds for all $n \leq n_0$, then $\Gamma$ is virtually nilpotent. So the added value here is that we need only one scale, where the doubling property occurs.
\end{remark}

We end this lecture by giving the following equivalent form of the main theorem:

\begin{theorem}[Structure theorem, equivalent form] \label{BGTbis} Let $G$ be a group, $K \geq 1$ a parameter and $A$ a finite subset with $|AA|\leq K|A|$. Then there is a subset $X \subset G$ such that $|X|\leq C(K)$ and a virtually nilpotent subgroup $\Gamma_0 \leq G$ such that $A \subset X\Gamma_0$.
\end{theorem}

The equivalence follows immediately from the following basic combinatorial lemma by taking $B=A \cap  \Gamma_0a$, where $\Gamma_0a$ is the coset obtained in Theorem \ref{BGT}.

\begin{lemma}[Ruzsa covering lemma]\label{ruzsa} Suppose $A,B$ are subsets in a group $G$ such that $|AB| \leq K|B|$, then there is a subset $X \subset A$ of size $\leq K$ such that $A \subset XBB^{-1}$.
\end{lemma}

\begin{proof} Pick a maximal family of disjoint translates of $B$ of the form $a_iB$ for $a_1, \ldots, a_N \in A$. Clearly $N \leq K$. Then for every $a \in A$, $aB$ intersects some $a_iB$ non trivially and hence $a \in a_iBB^{-1}$. Take $X=\{a_1,\ldots,a_N\}$.
\end{proof}

\section{The structure of approximate groups and geometric applications}

In the first lecture, we saw that our structure theorem \ref{BGT} easily implies Gromov's polynomial growth theorem. To provide further motivation for our study of the non-commutative inverse Freiman problem and of approximate groups, we begin this lecture with another geometric application.

\subsection{A generalized Margulis lemma}

In his book on metric structures on Riemannian and non-Riemannian spaces \cite{gromov-book}, Gromov discusses among many other things the properties of spaces with a lower curvature bound, in particular the Margulis lemma. The Margulis lemma is an important result in Riemannian geometry and hyperbolic geometry, where it is used for example to give lower bounds on the volume of closed manifolds, and to describe the geometry of cusps. Usually this lemma requires some curvature bounds on the manifold (see e.g. the books by Buser-Karcher \cite{buser-karcher} and Burago-Zalgaller \cite{burago-zalgaller}). The following result, which was conjectured by Gromov in this book, shows that the lemma holds in much greater generality.

\begin{theorem}[Generalized Margulis Lemma] Let  $X$ be a metric space and assume that every ball of radius $4$ can be covered by at most $K$ balls of radius $1$. Now let $\Gamma \leq Isom(X)$ be a discrete subgroup of isometries of $X$. Then there is $\eps_0>0$ depending on $K$ only and not on $X$ such that for every $x \in X$ the subgroup generated by
$$S_\eps(x):=\{\gamma \in \Gamma ; d(\gamma \cdot x,x)< \eps\}$$
is virtually nilpotent if $\eps<\eps_0$.
\end{theorem}

The classical Bishop-Gromov inequalities imply that hypotheses of this theorem are fulfilled for example when $X$ is the universal cover of a compact Riemannian $n$-manifold $M$ with the parameter $K$ depending only on $n$ and the lower bound on the curvature of $M$, and $\Gamma=\pi_1(M)$.

Note that there is no properness assumption on $X$. However when we say that $\Gamma$ is a ``discrete group of isometries'' we mean that the orbit maps $\Gamma \to X$, $\gamma \mapsto \gamma \cdot x$ are proper for every $x$, i.e. inverse images of bounded balls are finite. In fact all we need is that $S_2(x)$ is finite for each $x$.\\

\begin{proof} Clearly, from the triangle inequality we have: $S_2(x)^2 \subset S_4(x)$. Let us prove that $|S_2(x)^2|\leq K|S_2(x)|$. Write $B(x,4) \subset \cup_{i=1}^K B(x_i,1)$ for a collection of $K$ points $x_i \in X$. If $B(x_i,1)$ intersects $S_4(x)\cdot x$, then pick $\gamma_i \in S_4(x)$ such that $\gamma_i \cdot x \in B(x_i,1)$. Now given any $\gamma \in S_4(x)$, $\gamma \cdot x \in B(x_i,1)$ for some $i$, hence $d(\gamma \cdot x, \gamma_i \cdot x)\leq 2$. This means that $\gamma_i^{-1} \gamma \in S_2(x)$. Hence we proved $S_4(x) \subset XS_2(x)$ for some set $X\subset \Gamma$ with $|X|\leq K$. In particular $|S_2(x)^2|\leq |S_4(x)|\leq K|S_2(x)|$.

Now we can apply the structure theorem (here Theorem \ref{BGTbis}) to $S_2(x)$ and conclude that there is a virtually nilpotent group $\Gamma_0$ such that $S_2(x) \subset X\Gamma_0$ for some finite set $X$ with $|X|\leq C(K)$. Pick an integer $n$ with $C(K)< n < 2C(K)$ and $\eps_0=\frac{1}{n}$. If $\eps < \eps_0$, then by the triangle inequality $S_\eps(x)^n \subset S_2(x)\subset X\Gamma_0$. But $S_\eps(x)^k\Gamma_0/\Gamma_0$ is a non-decreasing sequence of sets and if for some $k$ $S_\eps(x)^k\Gamma_0/\Gamma_0=S_\eps(x)^{k+1}\Gamma_0/\Gamma_0$, then the sequence stabilizes and $S_\eps(x)^k\Gamma_0/\Gamma_0=\langle S_\eps(x)\rangle \Gamma_0/\Gamma_0$. Since $n>|X|$ it follows that $S_\eps(x)^n\Gamma_0/\Gamma_0=\langle S_\eps(x)\rangle \Gamma_0/\Gamma_0$ and $[\langle S_\eps(x)\rangle:\langle S_\eps(x)\rangle \cap \Gamma_0] \leq n \leq 2C(K)$. Hence $\langle S_\eps(x)\rangle$ is virtually nilpotent.
\end{proof}

Here again the nilpotent subgroup can be taken to have at most $O(\log K)$ generators and be of nilpotency class $O(\log K)$. \\

The best constant $\eps_0(X)$ is usually called the \emph{Margulis constant} of $X$.\\

Yet another application of the structure theorem is the fact, first proved by Gromov, that the fundamental groups of almost flat manifolds are virtually nilpotent. This was later extended by Fukaya-Yamaguchi to almost non-negatively curved manifolds and by Cheeger-Colding and Kapovich-Wilking to Ricci almost non-negatively curved manifolds. We can recover this last result from the structure theorem by a simple application of the Bishop-Gromov volume comparison estimates (see \cite{bgt})

\begin{theorem}[Almost non-negatively curved manifolds] Given $d$ there is $\eps(d)>0$ such that if $M$ is a Riemannian $d$-manifold with diameter $1$ and Ricci curvature bounded below by $-\eps(d)$, then $\pi_1(M)$ is virtually nilpotent.
\end{theorem}

\subsection{Approximate groups and the additive combinatorics toolkit}\label{tool}
We now return to additive combinatorics and describe a few combinatorial tricks and useful tools that one uses all the time when one deals with sets of small doubling. The first one is the so-called Ruzsa triangle inequality. Given two finite subsets $A,B$ in a group $G$ let

$$d(A,B):= \log \frac{|AB^{-1}|}{\sqrt{|A||B|}}$$
Clearly $d(A,B)=d(B,A)$ and $d(A,B) \geq 0$ because $|AB^{-1}|\geq |A|,|B|$. What is remarkable is the following:

\begin{lemma}[Ruzsa triangle inequality] Given three finite subsets $A,B,C$ in $G$, the triangle inequality holds
$$d(A,C) \leq d(A,B) + d(B,C)$$
\end{lemma}

Hence $d(A,B)$ is called the \emph{Ruzsa distance}\footnote{note that $d(A,B)=0$ does not imply $A=B$, but only (this is an exercise similar to Proposition \ref{extcase}) $A=aH$ and $B=bH$ for some finite subgroup $H$ and some $a,b \in G$.} between $A$ and $B$. The proof is very easy:\\

\begin{proof}Consider the map $AC^{-1} \times B \to AB^{-1} \times BC^{-1}$, which sends $(x,b)$ to $(a_xb^{-1},bc_x^{-1})$, where we made a choice of $a_x \in A$ and $c_x \in C$ for each $x \in AC^{-1}$. Then quite obviously this map is injective. Hence $|B||AC^{-1}|\leq |AB^{-1}||BC^{-1}|$, which is another way to phrase the triangle inequality $d(A,C) \leq d(A,B) + d(B,C)$.
\end{proof}

As a consequence:

$$\frac{|AA^{-1}|}{|A|} = e^{d(A,A)} \leq e^{d(A,A^{-1})+d(A^{-1},A)} = \big(\frac{|AA|}{|A|}\big)^2.$$

Using the triangle inequality one can also prove (\cite{tao,helfgott-sl2, petridis}):

\begin{lemma}[Small tripling lemma] If $A$ is any finite subset in a group $G$, and $n \geq 3$, then
$$\frac{|A^n|}{|A|} \leq \big( \frac{|A^3|}{|A|} \big)^{n-2}$$
\end{lemma}

Dealing with conditions like $|AA|\leq K|A|$, or $|A^{-1}A|\leq K|A|$ can be cumbersome at times and it is easy to run into unessential technical difficulties that hide the real ones. This is why Terry Tao came up with another set of axioms for a subset $A$, closely related to the small doubling condition, which is much easier to handle. He coined the following definition (\cite{tao}):

\begin{definition}[Approximate subgroup] Let $G$ be a group and $K \geq 1$ a parameter. A subset $A$ of $G$ is called an $K$-approximate subgroup if
\begin{itemize}
\item $A=A^{-1}$ and $1 \in A$,
\item $AA \subset XA$ for some subset $X \subset G$ with $|X|\leq K$.
\end{itemize}
\end{definition}

Note: this makes sense for $A$ infinite as well. But we will be mainly concerned with finite sets and by abuse of language $A$ will always be assumed finite in the sequel.\\

Clearly a $K$-approximate subgroup has doubling at most $K$. The converse is not true of course, but we have the following partial converse:

\begin{proposition}[Tao 2007 \cite{tao}]\label{approxlem} Suppose $A$ is a finite subset of an ambient group $G$ such that $|AA| \leq K|A|$. Then there is an $O(K^{O(1)})$-approximate subgroup $H$ of $G$ such that $A \subset XH \cap HY$ for some subsets $X,Y$ of size at most $O(K^{O(1)})$ and $|H|\leq O(K^{O(1)})|A|$.
\end{proposition}

The implied constants in this result are absolute. This proposition essentially reduces the study of large sets of bounded doubling to that of approximate groups. In a similar flavor one shows that if $|A^3| \leq K|A|$, then $(A \cup A^{-1} \cup \{1\})^2$ is a $O(K^{O(1)})$-approximate subgroup, and that if $|AA| \leq K|A|$, then $A$ has a subset of size $\geq |A|/O(K^{O(1)})$ which has tripling at most $O(K^{O(1)})$. \\

While the above facts require a bit of thought and are not immediate (see \cite{tao}), we can prove here a very easy variant:

\begin{lemma}\label{approx5} If $G$ is a group and $A$ a subset such that $A=A^{-1}$ and $|A^5| \leq K|A|$, then $A^2$ is a $K$-approximate subgroup.
\end{lemma}

\begin{proof}Simply apply Ruzsa's covering Lemma \ref{ruzsa}.
\end{proof}

It often happens that group theoretic arguments about finite subgroups of an ambient group can be transferred or adapted to approximate groups. This has been a guiding philosophy in much of our arguments. For example one has:

\begin{lemma}\label{subgroup} Suppose $A$ is a $K$-approximate subgroup and $B$ an $L$-approximate subgroup of an ambient group $G$ ($A$,$B$ not necessarily finite). Then $A^2 \cap B^2$ is a $(KL)^3$-approximate subgroup of $G$.
\end{lemma}

\begin{proof} Write $A^2 \subset XA$, $B^2 \subset YB$, so that $(A^2 \cap B^2)^2 \subset A^4 \cap B^4 \subset X^3A \cap Y^3B$. But every set of the form $xA \cap yB$ for $x \in X^3$, $y \in Y^3$ is contained in $z_{x,y}(A^2 \cap B^2)$ for some $z_{x,y}$ (pick any $z_{x,y}$ in $xA \cap yB$). So $(A^2 \cap B^2)^2 \subset Z(A^2 \cap B^2)$ for $Z=\{z_{x,y}\}_{x \in X, y\in Y}$ of size $\leq (KL)^3$.
\end{proof}

In particular if $H$ is a subgroup of $G$ and $A$ a $K$-approximate subgroup, then $A^2 \cap H$ is again a $K^3$-approximate subgroup.




\subsection{Nilprogressions and Tointon's theorem}

The structure theorem (Theorem \ref{BGT}) gave the key information about sets with bounded doubling, namely that they have a large intersection with a coset of a virtually nilpotent group. In order to get more precise information on the structure of these sets, or equivalently on the structure of approximate groups, we need to introduce the notion of \emph{nilprogression}. This generalizes to nilpotent groups the $d$-dimensional arithmetic progressions described above.\\

First we make the following general definition:

\begin{definition}[Progression in a general group] Let $G$ be a group and $x_1,\ldots,x_r$ be elements in $G$. Let $L_1,\ldots,L_r$ be non-negative integers. The set $P=P(x_1,\ldots,x_r; L_1,\ldots,L_r)$ of all group elements in $G$ which can be written as a word in the $x_i$'s and $x_i^{-1}$'s involving at most $L_i$ letters $x_i^{\pm 1}$ for each $i=1,\ldots, r$ is called a progression of rank $r$ and lengths $\{L_i\}_i$.
\end{definition}

For example if $L_1=\ldots = L_r=N$, then $P=P(x_1,\ldots,x_r; N,\ldots,N)$ is the ball of radius $N$ for the ``$\ell^\infty$-word metric'' with generating set $S:=\{x_1^{\pm 1},\ldots,x_r^{\pm 1}\}$, which we denote by $d_S^{\infty}$. Note that this is closely related to the usual word metric $d_S$ (or $\ell^1$-word metric), because obviously $d_S^{\infty} \leq d_S \leq r d_S^{\infty}$.\\

Similarly we define:

\begin{definition}[Nilprogression] A progression $P(x_1,\ldots,x_r; N,\ldots,N)$ is said to be a nilprogression with step $\leq s$ and rank $\leq r$ if the $x_i$'s generate  a nilpotent group with nilpotency class $\leq s$.
\end{definition}

and:

\begin{definition}[Coset nilprogression] A coset nilprogression is the inverse image of a nilprogression by a group homomorphism with finite kernel.
\end{definition}

Side remark: the ball-box principle (see \cite{gromov-cara}) regarding word balls in nilpotent groups (i.e. that they are sandwiched in between two ``nilpotent boxes'' of comparable size) remains valid for nilprogressions. This and some routine nilpotent algebra yields:

\begin{proposition}\label{nilapprox} There is $K=K(r,s)>0$ such that every coset nilprogression of rank $\leq r$,  step $\leq s$  and large enough side length (i.e. $L_i \geq C(r,s)$ for some constant $C(r,s)>0$) is a $K$-approximate subgroup.
\end{proposition}

For the proof, see either \cite[Prop. 2.2.3]{tao-hilbert} or \cite[Cor. 3.6]{breuillard-tointon}. Following previous work of Green-Ruzsa \cite{green-ruzsa} on the abelian case, Green and the author \cite{breuillard-green} on torsion-free nilpotent groups, and Tao \cite{tao} on $2$-step groups, M. Tointon recently generalized the Freiman-Ruzsa-Chang theorem (Theorem \ref{freiman}) to an arbitrary nilpotent group. Using a clever lifting of approximate subgroups and induction on nilpotency class he proved:

\begin{theorem}[Tointon's theorem \cite{tointon}] Let $A$ be a $K$-approximate subgroup of a nilpotent group $G$ with nilpotency class $s$. Then there is a coset nilprogression $P$ with rank $\leq r=O_s(K^{C(s)})$ and step $\leq s$ such that $A \subset P \subset XA$ for some $X \subset G$ of size $|X|\leq O_s(K^{C(s)})$.
\end{theorem}

\subsection{Stronger form of the structure theorem}

We now are in a position to state the stronger form of our structure theorem.

\begin{theorem}[BGT strong form]\label{strong} Suppose $A$ is a $K$-approximate subgroup in an ambient group $G$. Then $A \subset XP$, where $X \subset G$ has size $|X|\leq C(K)$ and $P$ is a coset nilprogression whose rank and step are bounded in terms of $K$ only and such that $|P|/|A|$ is bounded in terms of $K$ only.
\end{theorem}

\noindent {\bf Remarks:}

a) Here again $C(K)$ is some constant depending only on $K$, but the proof gives no explicit estimate in terms of $K$.  It can in fact be shown that $C(K)$ is not upper-bounded by $K^C$ for some absolute constant $C>0$. See \cite{breuillard-tointon} for an example.

b) Below we derive the above strong form from the weak form given earlier in Theorem \ref{BGTbis} using Tointon's theorem. Originally  in \cite{bgt} we proved the strong form directly (thus obtaining as a by-product a new proof of the original abelian Freiman-Ruzsa theorem) and obtained explicit bounds for the rank and step of the nilprogression and its size as well, but not on the constant $C(K)$. In fact we show in \cite{bgt} that one can find a coset nilprogression $P$ as in Theorem \ref{strong} with $P \subset A^{12}$ (and hence $|P|\leq K^{11}|A|$) whose rank and step are $O(K^3)$.\\

\begin{proof} From the weak form of the structure theorem, more precisely Theorem \ref{BGTbis}, we can find a virtually nilpotent subgroup $\Gamma_0 \leq G$ such that $A \subset X\Gamma_0$ for some subset $X$ of size $|X|\leq C(K)$. So  $|A^2 \cap \Gamma_0|\geq |A|/C(K)$. In fact (see Remark d) after Theorem \ref{BGT}) the proof shows that $\Gamma_0$ fits into an exact sequence $1 \to N \to \Gamma_0 \overset{\pi}{\rightarrow} L \to 1$, where $L$ is nilpotent with bounded rank and step, and $N$ is finite with $N \subset XA$ for some bounded set $X$, $|X|\leq c(K)$. By Lemma \ref{subgroup} $A^2 \cap \Gamma_0$ is a $K^3$-approximate subgroup.  Then $\pi(A^2 \cap \Gamma_0)$ is a $K^3$-approximate subgroup of $L$ and Tointon's theorem applies: there is a coset nilprogression $Q$ with bounded rank and step which contains $\pi(A^2 \cap \Gamma_0)$ and is contained in boundedly many translates of $\pi(A^2 \cap \Gamma_0)$. Set $P=\pi^{-1}(Q)$. This is again a coset nilprogression with bounded rank and step. It contains $A^2 \cap \Gamma_0$, so by Ruzsa covering (Lemma \ref{ruzsa}) $A$ lies in $XP$ for some bounded $X$. Moreover, $P$ is contained in boundedly many translates of $A$, because $N$ is.
\end{proof}

We now give an interesting consequence of the strong form of the theorem and in the last lecture, we will give two geometric applications, one to Cayley graphs of finite simple groups, the other to scaling limits of vertex transitive graphs, that both require this strong form.

  Suppose
 $$P=P(x_1,\ldots,x_r;L_1,\ldots,L_r)$$
 is a nilprogression of rank $r$ and step $s$, then for every $n \in \N$ the power set $P^n$ is obviously contained in $P(x_1,\ldots,x_r;nL_1,\ldots,nL_r)$. Using some nilpotent algebra it can be shown (by arguments similar to those in the proof of Proposition \ref{nilapprox}) that if the $L_i$'s are large enough, it also contains $P(x_1,\ldots,x_r;\frac{n}{C}L_1,\ldots,\frac{n}{C}L_r)$ for some $C=C(r,s)>0$ and that the three sets are of comparable size up to multiplicative constants depending only on $r$ and $s$. In particular $|P^{2n}|\leq K|P^n|$ for some constant $K=K(r,s)$ independent of the $L_i$'s and independent of $n$. So the bound on the doubling $P$ is inherited to all powers $P^n$. From this observation (proved in detail in \cite{breuillard-tointon}) and the strong form of the theorem it is then an exercise (similar to the derivation of Theorem \ref{gromov} from Theorem \ref{BGT}) to derive:

\begin{theorem}[doubling at one scale $\Rightarrow$ doubling at all scales, \cite{breuillard-tointon}]\label{bt-thm} Given $K>0$ there is $n_0,C(K)>0$ such that if $S$ is a finite subset of a group $G$ such that $S^{n_1}$ is a $K$-approximate subgroup for some $n_1 > n_0$, then $S^n$ is a $C(K)$-approximate subgroup for every $n \geq n_1$.
\end{theorem}

\section{Hilbert's fifth problem and approximate groups}

In this third lecture, we outline the proof of the structure theorem (Theorem \ref{BGT}). A good deal of this lecture is devoted to explaining the key step of the Gleason-Yamabe theorem on the structure theory of locally compact groups.

\subsection{Introduction} Approximate subgroups of specific groups (such as linear groups, free groups, finite simple groups, etc...) had been studied by various authors relying on techniques specific to the given class of groups they were considering, often yielding explicit bounds. But no approach to tackle the general problem of understanding the structure of approximate subgroups of arbitrary groups existed prior to the work of Hrushovski \cite{hrushovski}.

In his paper \cite{hrushovski}, Hrushovski proposed an approach based on an analogy between this question and various problems and results from Stability Theory, a branch of Model Theory. His theorem, which we state below in a form that avoids the language of model theory, is the first step of the proof. For readers interested in the connection with model theory, we recommend van den Dries' Bourbaki report \cite{DriesBourbaki} as well as Hrushovski's own lecture notes.

Combined with the Gleason-Yamabe structure theorem for locally compact groups (Hilbert 5th problem) Hrushovski's theorem already gives non trivial information on the structure of an arbitrary approximate group. Roughly speaking Hrushovski derives from it a weaker structure theorem than Theorem \ref{BGT} in which approximate groups are showed to admit a nested sequence $S_{n+1}^2 \subset S_n$ of large approximate subgroups $S_n$ of $A$, which exhibits the following weak form of nilpotence: $[S_n,S_n] \subset S_{n+1}$, see \cite[Theorem 1.1]{hrushovski}.

In order to obtain the finer structure given by our Theorem \ref{BGT}, one has to go inside the proof of the Gleason-Yamabe theorem and adapt arguments discovered there, more than sixty years ago, in the context of topological groups to the setting of approximate groups.

Going through the proof of the Gleason-Yamabe theorem in the original papers of Gleason and of Yamabe or in the subsequent exposition in the book by Montgomery and Zippin can be daunting. Nevertheless the task has been greatly eased by a number of people over the years, including I. Kaplansky \cite{kaplansky} and especially J. Hirschfeld \cite{hirschfeld} who identified and formulated the key claims (the Gleason-Yamabe lemmas) and gave a much streamlined treatment of the proof using ultra-products (recently simplified further by model theorists  L. van den Dries and I. Goldbring in a forthcoming article \cite{goldbring-van-den-dries}). The whole thing is now pretty demystified, and a clear and complete modern treatment is now available in Tao's book \cite{tao-hilbert}.

\subsection{Proof strategy}

We turn to the proof of the structure theorem. As in Gromov's proof of his theorem \cite{gromov} on groups with polynomial growth, one needs to pass from a discrete object (in our case a finite approximate group) to an infinite object on which we can perform analysis. While Gromov used his notion of Gromov-Hausdorff convergence to define his limit space (i.e. the asymptotic cone = the Cayley graph viewed from infinity) van den Dries and Wilkie \cite{dries-wilkie} later suggested to construct this space using ultra-filters. We will follow the same route and start with a non-principal ultrafilter $\mathcal{U}$ on the set of positive integers.

Fix $K \geq 1$. Given any family of groups $G_n$ and a family of finite $K$-approximate subgroups $A_n \leq G_n$, we form the ultra-product:

$$\A := \prod_{\mathcal{U}} A_n, \textnormal{ } \G:=\prod_{\mathcal{U}} G_n.$$

The first observation is that $\A$ is a $K$-approximate subgroup of $\G$. Indeed the ultra-product of a sequence of sets of cardinal at most $K$ has cardinal at most $K$. We call such a set $\A$ an \emph{ultra approximate group}.\\

The proof has essentially three main steps:\\

Step 1 : Build a topology.

Step 2 : Adapt the Gleason-Yamabe lemmas.

Step 3 : Conclude.\\

Step 1: This step is due to Hrushovski \cite{hrushovski}. While Gromov started with a sequence of metric spaces (Cayley graphs), we do not have any distance given a priori with our sequence of sets $A_n$. So the first goal here is to define a notion of closeness or nearness for elements of $A$. One would like to say that an element $g$ is close to the identity if the symmetric difference $gA \Delta A$ is very small in size. While this does not quite work on the nose, a variation of this idea can implemented successfully. This ``distance'' being defined for each $A_n$, it passes to the ultra-product to give rise to a genuine left invariant pseudo-distance $d(\cdot, \cdot)$ on the group $\langle \A \rangle$. One then shows that $G:=\langle \A \rangle/ \{ \a ; d(\a,1)=0\}$ is a locally compact second countable group. \\

Step 2: This part consists in adapting the Gleason-Yamabe lemmas from locally compact groups to approximate groups. The Gleason-Yamabe lemmas lie at the heart of the proof of the structure theorem for locally compact groups proved by Gleason and Yamabe in the early 1950's. According to this theorem, every locally compact group has an open subgroup which is a projective limit of Lie groups, i.e. has arbitrarily small normal compact subgroups modulo which it is a Lie group with a finite number of connected components. The Gleason-Yamabe lemmas serve two purposes in the proof: they allow to reduce to the NSS case (NSS= no small subgroups) and to show that the set of one-parameter subgroups of an NSS group is a vector space, thus giving rise to a Lie algebra and a Lie structure. It turns out that under some assumptions on the finite approximate group $A$ (which are anyway satisfied by a large approximate subgroup of $A$ as follows from Step 1), these lemmas can be adapted almost word for word to $A$ and yield  a finite subgroup $H$ of $\langle A \rangle$ contained in $A$ and normalized by $A$, together with an element $a \in A \setminus H$ which commutes with $A$ modulo $H$. This element will serve as a stem from which to build the nilprogression.\\

Step 3: Performing Step 2 for each $A_n$, we obtain an ultra finite subgroup $\HH:=\prod_{\mathcal{U}} H_n \vartriangleleft \langle \A \rangle$ and an ultra cyclic subgroup $\langle \a \rangle \leq \langle \A \rangle$, where $\a$ is the class of $(a_n)_{n}$. By Step 1 and the structure theorem for locally compact groups, we may assume after passing to an open subgroup that $G$ is a projective limit of Lie groups. Moding out by the maximal normal compact subgroup of $G$, we obtain a connected Lie group $L$. If $\dim L=0$, then $\A$ is covered by finitely many translates of a compact open subgroup and this means $A_n$ is covered by boundedly many translates of a finite subgroup for $\mathcal{U}$-almost every $n$. If $\dim L>0$, the projection of $\langle \a \rangle$ to $L$ is a central one-parameter subgroup. The rest of the proof proceeds by induction on $\dim L$, moding out $\langle \A \rangle$ by $\langle \a \rangle$ and starting all over again. A careful treatment of this step requires the introduction of \emph{local groups}\footnote{they can be avoided if one only shoots for Theorem \ref{BGT}, see the proof in \cite{DriesBourbaki}, but seem necessary for the finer structure described in Remark d) after Theorem \ref{BGT}, and hence for the strong form Theorem \ref{strong}.}. At the end we obtain that the $A_n$'s are covered by boundedly many translates of a finite-by-nilpotent subgroup, the finite part being itself contained in boundedly many translates of $A_n$. The theorem then follows by contradiction, picking a sequence of potential counter-examples $A_n$ to the statement of the theorem and carefully negating the quantifiers. \\

\noindent N.B. In Gromov's theorem the induction in the final step is done on the exponent of polynomial growth. In his paper \cite{hrushovski}, Hrushovski also gave another proof of Gromov's theorem and even a strengthening of it (which is now a simple consequence of Theorem \ref{BGT}) using induction on $\dim L$ as in Step 3 above. The novelty of \cite{bgt}, say compared to \cite{hrushovski}, consists mainly in Step 2.\\

Below we describe Step 1 and Step 2 and refer the reader to the original paper for Step 3, whose details are a bit more technical.

\subsection{Building a locally compact topology (Step 1)}

The goal here is to define a non-trivial topology on $\langle \A \rangle$ which, when a certain normal subgroup is quotiented out, turns $\langle \A \rangle$ into a locally compact group. We need to define a base of neighborhoods of the identity. The main tool for this is the following lemma:

\begin{lemma}[Sanders \cite{sanders}, Croot-Sisask \cite{croot-sisask}, Hrushovksi \cite{hrushovski}] If $A$ is a $K$-approximate group and $k \geq 1$, then there is $S \subset A^4$ such that
\begin{itemize}
\item $1 \in S$, $S=S^{-1},$
\item $S^k \subset A^4$,
\item $|S|\geq |A|/C(K,k)$
\end{itemize}
\end{lemma}

In other words, this lemma allows to take a ``square root'', or $k$-th root, of any given approximate group. The lemma was proved independently by Sanders using a combinatorial argument (showing the existence of a subset $B \subset A$ such that there is a non-trivial proportion of elements $g \in A^4$ such that $|gAB \Delta AB|$ is at most $\frac{1}{k}|AB|$ say), by Croot-Sisask (who use the probabilistic method to find the subset $S$) and by Hrushovski (using tools from Model Theory). \\

Using this lemma, we can define a base of neighborhoods of the identity in $\langle \A \rangle$ as follows. For each $A_n$ and each $k \geq 1$ pick a subset $S_{n,k}$ as given by the above lemma. In fact a rather straightforward strengthening of the above lemma ensures that we can pick $S_{n,k}$ in such a way that $S_{n,k+1}^2 \subset S_{n,k}$ and $gS_{n,k+1}g^{-1} \in S_{n,k}$ for all $g \in A$.\\

Set $\SSS_k:= \prod_{\mathcal{U}} S_{n,k}$. This defines a topology on $\langle \A \rangle$ with the $\SSS_k$ as a base of neighorhoods of the identity. Moreover the intersection $\Gamma:= \cap_{k \geq 1} \SSS_k$, where $\SSS_k:= \prod_{\mathcal{U}} S_{n,k}$ is a normal subgroup of $\langle \A \rangle$. We define $G$ as the quotient of $\langle \A^4 \rangle/\Gamma$. In this topology $\A$ mod $\Gamma$ is a compact neighborhood of the identity.\\

Given a Hausdorff topological group with a nested sequence of symmetric neighborhoods of the identity $S_{i+1}^2 \subset S_i$ a classical construction of Birkhoff and Kakutani \cite{birkhoff, montgomery-zippin} yields a left invariant distance compatible with the topology (in which the $S_i$ are essentially the balls of radius $\frac{1}{2^i}$ around the identity). We can do this here with the $\SSS_k$'s and obtain a left invariant distance $d$ on $G$.\\

Then the local compactness of $G$ is proven in the usual way by establishing completeness with respect to $d$ and total boundedness (i.e. $\forall \textnormal{ }0< \eps < r$ every ball of radius $r$ can be covered by $\leq C(\eps,r)$ balls of radius $\eps$). This follows from the Ruzsa covering lemma. Completeness follows by the countable saturation property of ultrafilters. \\

In the Birkhoff-Kakutani construction of the left invariant distance $d$, balls can be approximated by finite products of the $S_i$'s. Since the $\SSS_k$'s are ultra-product sets, i.e. given by the ultra-product of subsets of $A_n$, we obtain that balls in $G$, and more generally any compact subset of $G$, can be approximated by ultra-product sets.

We summarize the above discussion in the following statement:

\begin{theorem}{(Hrushovski local model theorem, \cite[Theorem 4.2.]{hrushovski})}\label{hrushov} There exists a locally compact group $G$ and a group homomorphism
$$\pi: \langle \A \rangle \to G$$
such that $\pi(\A)$ is a compact neighborhood of the identity in $G$, and for every $F \subset O \subset \pi(\A^4)$ with $F$ closed and $O$ open subsets, there are $B_n \subset A_n^4$ such that
$$\pi^{-1}(F) \subset \B \subset \pi^{-1}(O),$$
where $\B:=\prod_{\mathcal{U}} B_n$.
\end{theorem}

\noindent { \bf Remarks.}

a) If $G$ is the trivial group, then the theorem implies that the $A_n$'s are finite groups for almost all indices $n$ w.r.t the ultrafilter $\mathcal{U}$. Indeed take $O=F=\{1\}$. Then according to the theorem there is $\B= \prod_{\mathcal{U}} B_n \subset \A$ such that $B = \ker \pi = <A>$. Thus A is stable under multiplication. Hence so are the $A_n$ for $\mathcal{U}$-almost every $n$.

b) Similarly if the ``model group'' $G$ is totally disconnected, then we can conclude that there is a finite group $H_n \subset A_n^{C}$ for some bounded $C$ such that $A_n$ is contained in boundedly translates of $H_n$ (exercise: this implies that $H_n$ and $A_n^2$ are commensurable). Indeed a totally disconnected locally compact group contains an open compact subgroup (van Dantzig). According to the theorem it must be an ultra-product set $\HH:=\prod_{\mathcal{U}} H_n$ and the $H_n$ must be stable under multiplication for $\mathcal{U}$-almost every $n$, i.e. be finite subgroups. Since $\A$ is compact in projection to $G$, it must be finitely covered by $\HH$.

c) Recall the Gleason-Yamabe structure theorem for locally compact groups:

\begin{theorem}[Gleason-Yamabe \cite{gleason-annals,yamabe2}] \label{gleason-yamabe} If $G$ is a locally compact group, then $G$ has an open subgroup $G'$ which is a ``generalized Lie group'' in the sense that every identity neighborhood contains a compact subgroup $H \vartriangleleft G'$ such that $G'/H$ is a Lie group with finitely many connected components.
\end{theorem}

N.B. If $G$ is compact, the above boils down to the Peter-Weyl theorem (which is itself an application of the spectral theorem for self-adjoint convolution operators acting on $L^2(G)$). It is worth noting that the proof of Theorem \ref{gleason-yamabe} for non-compact groups is not independent of the Peter-Weyl theorem as it makes use of it in the crucial reduction step to the NSS (No Small Subgroup) case.\\

A consequence of this is that attached to every non-compact locally compact group there is a canonically defined non-compact connected Lie group $L$ with no non-trivial compact normal subgroups. Indeed $L$ is the connected component of the identity of $G'/H$ modulo its maximal normal compact subgroup. Its isomorphism class is independent of $G'$ and $H$. The connected component $G^0$ is compact if and only if $L$ is trivial.\\

In the setting of Theorem \ref{hrushov}, we call $L$ the \emph{Hrushovski Lie group} associated to the $A_n$'s. If it is trivial, then the $A_n^2$ are commensurable to finite groups (this is Remark b) above). In order to understand what happens if $L$ is non-trivial, we need to go into the proof of the Gleason-Yamabe theorem. This is the aim of Step 2. At the end (Step 3) we establish that $L$ is a simply connected nilpotent Lie group.

\subsection{The Gleason-Yamabe theorem}

 Let us digress for a little while  and discuss the proof of the Gleason-Yamabe structure theorem for locally compact groups (Theorem \ref{gleason-yamabe}). At the heart of it lie the so-called Gleason-Yamabe lemmas (Lemmas \ref{GY1} and \ref{GY2} below). They serve a dual purpose:

 \begin{enumerate}
 \item reduce the structure theorem to the NSS case (i.e. the case when $G$ has an identity neighborhood with no non-trivial subgroup).
 \item show that if $G$  is NSS, the set of one-parameter subgroups (i.e. homomorphisms from $(\R, +)$ to $G$) forms a vector space, thus giving rise to the sought-after Lie algebra and Lie group structure.
 \end{enumerate}

Historically, Gleason \cite{gleason-annals} used a construction closely related to Lemma \ref{GY2} below to prove Theorem \ref{gleason-yamabe} in the case of NSS connected locally compact groups of finite topological dimension. It was Yamabe \cite{yamabe1, yamabe2} who first understood how to deal with the general (non NSS) case by reducing it to the Peter-Weyl theorem using a variant of Gleason's argument (encapsulated in Lemma \ref{GY1}). We stress that there are two important cases in which the theorem is much easier to prove: the compact case (Peter-Weyl), and the totally disconnected case (van Dantzig).\\

Suppose $G$ is a locally compact group and pick a compact identity neighborhood $V$ in $G$.
Given a subset $Q \subset V$, let $n_V(Q)$ be the largest integer such that $1,Q,\ldots,Q^{n-1}$ are all contained in $V$. Clearly $n_V(Q) >1$ iff $Q \subset V$, and $n_V(Q)=+\infty$ iff the subgroup $\langle Q \rangle$ lies in $V$ entirely. Set $n_V(g)=n_V(\{g\})$.

\begin{lemma}[Gleason-Yamabe 1]\label{GY1} There is $C=C_V>0$ such that for every $K \geq 1$, there is an identity neighborhood $U \subset V$, such that for every subset $Q \subset U$,
\begin{itemize}
\item either $n_{V^4}(Q) \geq K n_V(Q),$
\item or there is $g \in Q$ such that $n_U(g) \leq KC_V n_V(Q)$.
\end{itemize}
\end{lemma}

In other words, unless it takes a very long time for powers of $Q$ to venture outside of $V^4$ after they first pop out of $V$, the time it takes for powers of $Q$ to escape $V$ is comparable to the maximal time needed for powers of individual members of $Q$ to escape $U$.


A direct consequence of this is:
\begin{corollary}[subgroup trapping]\label{trapping} Let $G$ be a locally compact group. For every identity neighborhood $V$ in $G$, there is an identity neighborhood $U \subset V$ such that the subgroup $H$ generated by all subgroups contained in $U$ is contained in $V$.
\end{corollary}

\begin{proof} Take for $Q$ the union of all subgroups contained in $U$ and apply Lemma \ref{GY1}.
\end{proof}


\begin{proof}[Proof of Theorem \ref{gleason-yamabe} : reduction to the NSS case.] We claim that there exists an open subgroup $G'$ in $G$ and, for each compact neighbourhood $V$ of the identity in $G'$, a normal subgroup $N' \lhd G'$ such that $N' \subset V$ and such that $G'/N'$ is NSS. This is clearly sufficient to reduce Theorem \ref{gleason-yamabe} to the NSS case.

Up to passing to an open subgroup, we may assume that $G$ is almost connected, in the sense that its connected component of the identity is co-compact. Van Dantzig's theorem ensures that this is possible: indeed, every totally disconnected locally compact group has an open compact subgroup, so the pre-image in $G$ of a compact open subgroup of $G/G^0$, where $G^0$ is the connected component of $G$, is an almost connected open subgroup of $G$.

Let $V$ be a compact neighborhood of the identity in $G$. By Corollary \ref{trapping}, there is an open neighborhood of the identity $U \subset V$ such that the closure $H$ of the subgroup generated by all subgroups contained in $U$ is contained in $V$. Since $H$ is a compact group, we may apply the Peter-Weyl theorem to it. In particular, we may find a closed normal subgroup $N \vartriangleleft H$ contained in $U$ such that $H/N$ is NSS. This means that there is an open neighborhood of the identity $W_0$ in $G$ containing $N$ such that every subgroup of $H$ containing $N$ and contained in $W_0$ is in fact a subgroup of $N$. We can pick a smaller identity neighborhood $W$ of $G$, which contains $N$, and is such that $W^4 \subset W_0$. Then we see that every subgroup of $H$ contained in $W^3$ must lie in $N$, and in particular $gNg^{-1} \subset N$ for every $g \in W$, because $gNg^{-1}$ lies in $W^3 \subset U$ so $gNg^{-1}$ is indeed a subgroup of $H$. This means that $N$ is normalised by $\langle W \rangle$  and that $\langle W \rangle /N$ is NSS.

However $\langle W \rangle$ is an open subgroup of the almost connected group $G$. Hence it has finite index in $G$ (it contains the connected component $G^0$ and is open modulo $G^0$ in the profinite group $G/G^0$).  Thus there are only finitely many conjugates $g_iNg_i^{-1}$ of $N$ in $G$. So $N':=\cap g_iNg_i^{-1}$ is normal in $G$ and contained in the neighborhood of the identity $W':=\cap g_iWg_i^{-1}$. Moreover any subgroup contained in $W'$ is contained in $N'$, so $G/N'$ is NSS.
\end{proof}

\begin{proof}[Proof of  Corollary \ref{trapping}] We now derive Corollary \ref{trapping} from Lemma \ref{GY1}. This uses Peter-Weyl's theorem one more time. Let $B(t)$ be the closed ball of radius $t$ around the identity (for a fixed left invariant distance), let $Q_t$ be the union of all subgroups contained in $B(t)$. Fix $s>0$. For $t<s$, $Q_t \subset B(s)$, and $n_{B(s)}(g)=+\infty$ for $g \in Q_t$. So Lemma \ref{GY1} implies that either $n_{B(s)}(Q_t)=+\infty$ in which case $\langle Q_t \rangle$ is relatively compact and we are done, or $\frac{n_{B(4s)}(Q_t)}{n_{B(s)}(Q_t)}$ is unbounded as $t \to 0$. This implies that $Q_t^{n_{B(s)}(Q_t)}$ converges along a sequence of $t \to 0$ towards a compact subgroup $H_s \subset B(4s)$ not contained in $B(s/2)$. By diagonal process we obtain this way a nested family of groups $H_s \leq H_{s'}$ if $s\leq s'$. But by Peter-Weyl, since $H_1$ is compact, it has a closed normal subgroup $N_1$ such that $H_1/N_1$ is NSS. Now doing the above construction with $B(s)N_1$ in place of $B(s)$ we obtain a nested family of subgroups $H'_s \subset (B(s)N_1)^4 \cap H_1 \subset B(4s)N_1 \cap H_1$ such that $H'_s$ is not contained in $B(s/2)N_1$. This contradicts the fact that $H_1/N_1$ is NSS. \end{proof}

We now pass to the second part of the Gleason-Yamabe lemmas and thus assume that $G$ is NSS. It is then easy to see (arguing by contradiction and taking a limit) that $n_U$ and $n_V$ are comparable if $U$ and $V$ are any two small enough identity neighborhoods, i.e.
\begin{equation}\label{comp}
\frac{1}{C}n_U(Q) \leq n_V(Q) \leq C n_U(Q)
\end{equation}
for some $C=C(U,V)>0$ and all subsets $Q$. We then define for $g \in G$
$$\|g\|_U:= \frac{1}{n_U(g)},$$
and call it the \emph{escape norm} associated to $U$.

\begin{lemma}[Gleason-Yamabe 2]\label{GY2} Suppose $G$ is a locally compact group which is NSS and $V$ is a neighborhood of the identity. Then $V$ contains a compact identity neighborhood $U$ such that, for some $C=C_U>0$, and all $g,h \in U$,
\begin{enumerate}
\item $\|h^{-1}gh\|_U \leq C \|g\|_U$
\item $\|gh\|_U \leq C(\|g\|_U + \|h\|_U)$
\item $\|[g,h]\|_U \leq C \|g\|_U \cdot \|h\|_U$
\end{enumerate}
\end{lemma}

Item 1) is clear from the definition and $(\ref{comp})$. Item 2) is an immediate consequence of the first Gleason-Yamabe lemma, i.e. Lemma \ref{GY1}, and of $(\ref{comp})$ (take $Q=\{g,h\}$). Item 3) can be proved by an argument similar that used in the proof of Lemma \ref{GY1} (see below).\\

Lemma \ref{GY2} is used to prove that the set $L(G)$ of one-parameter subgroups of $G$ forms a vector space on which (an open subgroup of) $G$ acts by conjugation with central kernel. From this point the end of the proof of Theorem \ref{gleason-yamabe} is fairly standard: one shows that $L(G)$ is locally compact, hence finite dimensional, thus (an open subgroup of) $G$ embeds in $\GL(\R^n)$ modulo the center and one concludes by the Cartan-von-Neumann theorems that $G$ is a Lie group.

The NSS assumption easily implies (taking the cyclic group generated by group elements closer and closer to the identity and passing to a limit) that $L(G)$ is non-trivial. In order to prove that $L(G)$ is a vector space, one needs to make sense of the formula:

$$\lim_{n \to +\infty} \left(X(\frac{t}{n}) Y(\frac{t}{n})\right)^n = (X + Y)(t)$$
given two one-parameter subgroups $X,Y \in L(G)$. This is where Item 2) of Lemma \ref{GY2} comes into play, because it allows to conclude that $(X(\frac{t}{n}) Y(\frac{t}{n}))^n$ remains in a fixed compact part of $G$ when $n$ tends to $+\infty$ and $t$ remains bounded. In fact, using Item 3), i.e. the commutator shrinking property, one shows that this is a Cauchy sequence (NSS groups are metrisable) and hence converges.

Finally, using Item 3) again, one proves that every small enough element of $G$ lies on a one-parameter subgroup. Then after passing to an open subgroup of $G$, we see that the kernel of the adjoint representation is central. This concludes the proof of Theorem \ref{gleason-yamabe}.\\

\noindent \emph{The proof of the Gleason-Yamabe lemmas.} Having explained how the Gleason-Yamabe theorem is derived from the Gleason-Yamabe lemmas, let us now say a few words about the proof of the lemmas themselves. As said earlier Items 1) and 2) follow easily from the NSS condition and Lemma \ref{GY1}, so we are left with proving Lemma \ref{GY1} and Item 3) of Lemma \ref{GY2}.

Before we start, it is instructive to notice that if $G$ were a Lie group, then both statements would be easy to prove using the exponential map $\exp: L(G) \to G$, as $\|\exp(X)\|_U$ is then comparable to $||X||$ for any fixed norm on $L(G)$. But we are precisely trying to establish the existence of an exponential map...

We now focus on Lemma \ref{GY1}. There are two bright ideas involved here: the first, say in the setting of Lemma \ref{GY1}, is to introduce a continuous non-negative bump function $\phi$, which vanishes outside $V^4$ and is equal to $1$ at the identity, and compare $\|g\|_U$ to the sup norm $||\partial_g \phi||_\infty$, where $\partial_g \phi(x):=\phi(g^{-1}x) - \phi(x)$. We make the following simple observations:

\begin{enumerate}
\item We have the triangle inequality $||\partial_{gh} \phi ||_\infty \leq ||\partial_g \phi ||_\infty + ||\partial_h \phi ||_\infty$,
\item If $||\partial_g \phi ||_\infty < 1$, then $g \in Supp(\phi) \subset V^4$.
\end{enumerate}
from which it follows immediately that $||g||_V \leq ||\partial_g \phi||_\infty$. We cannot expect that an inequality in the other direction holds for all $g$, because that would imply that $\phi$ is invariant under every subgroup contained in $V$. But, for every  $K \geq 1$ we will find $U$ small enough and build a function $\phi$ (depending on $K$,$V$) supported in $V^4$, such that $\phi(1)=1$,  and for which the following inequality holds for all  $Q \subset U$, all $g \in Q$, and  some $C_V>0$
\begin{equation}\label{eqq}||\partial_g \phi||_\infty \leq  C_V\frac{1}{n_U(g)} +  \frac{1}{2K }\cdot\frac{1}{n_V(Q)}.\end{equation}
Lemma \ref{GY1} follows easily. Indeed, by definition of $n_{V^4}(Q)$ there is a word $w$ of length $n_V(Q)$ with letters in $Q$ such that $w \notin V^4$.  Hence $||\partial_w \phi||_\infty \geq 1$ by $(2)$ above, and by the triangle inequality $(1)$ there must be $g \in Q$ such that $||\partial_g \phi||_\infty \geq \frac{1}{n_{V^4}(Q)}$. Suppose $n_{V^4}(Q)\leq K n_V(Q)$, then we conclude from $(\ref{eqq})$ that $n_U(g) \leq KC_V n_V(Q)$ as desired.\\

So things boil down to finding $\phi$ such that $(\ref{eqq})$ holds. If we were on a Lie group and $\phi$ were say of class $C^2$, because the second derivatives are controlled, we would have $\partial_{g^n} \phi \simeq n \partial_g \phi$ as long as $g^n \in U$ if $U$ is small enough. This follows immediately from the Taylor expansion of $\phi$ near the identity, or from the following abstract form of the Taylor formula, which always holds (direct check):
\begin{equation}\label{taylor}
\partial_{g^n} \phi = n \partial_g \phi + \sum_{i=0}^{n-1}  \partial_{g^i} \partial_g \phi
\end{equation}
The second term in the right hand side would then be negligible and it would follow that $||\partial_g \phi||_\infty \leq C||g||_U$, which is a strong form of $(\ref{eqq})$.\\

This is where the second bright idea comes in: making use of a (left-) Haar measure, we can define $\phi$ as the convolution of two bump functions $\phi(x):=\frac{\psi_1 *\psi_2(x)}{\psi_1 *\psi_2(1)}$, where $\psi_1 * \psi_2(x)=\int \psi_1(xy)\psi_2(y^{-1})dy$. The following simple formula allows to control the second derivatives of $\phi$ in terms of the first derivatives of $\psi_1$ and $\psi_2$:

$$\partial_h \partial_g (\psi_1 * \psi_2) = \int (\partial_g \psi_1)(y) \cdot (\partial_{h^y} \psi_2)(y^{-1}x) dy.$$
where $h^y=y^{-1}hy$. We take
$$\psi_1(x):=(1 - \frac{1}{n_V(Q)} \inf \{n\geq 0, x \in Q^nV\})^+$$
and
$$\psi_2(x):=(1-\frac{1}{L} \inf \{n\geq 0, x \in W^nV\})^+,$$ where $y^+:=\max\{0,y\}$ and $W$ is an identity neighborhood such that $W^{L} \subset V$ for some integer $L > 2K vol(V^3)/vol(V \cap V^{-1})$.

It is clear that $\psi_1$ and $\psi_2$ are equal to $1$ on $V$ and their supports are contained in $V^2$. Hence the support of $\phi$ lies in $V^4$, $||\psi_1 * \psi_2|| \leq vol(V^2)$ and $\psi_1*\psi_2(1) \geq vol(V \cap V^{-1})$.  Hence

$$||\phi||\leq \frac{vol(V^2)}{vol(V \cap V^{-1})}$$

We also have Lipschitz control on both $\psi_1$ and $\psi_2$, because if $g \in Q$, then $||\partial_g \psi_1 || \leq \frac{1}{n_V(Q)}$, while if $h \in W$, then $||\partial_h \psi_2 || \leq \frac{1}{L}$. Moreover if $g \in V$, then the support of $\partial_g \psi_1$ is contained in $V^3$.

Setting $U = \cap_{y \in V^3} y^{-1} W y$, we conclude that for all $g \in Q$ and $h \in U$:

\begin{equation}\label{dou}
||\partial_h \partial_g (\psi_1 * \psi_2) || \leq vol(V^3) \frac{1}{n_V(Q)} \frac{1}{L}.\
\end{equation}
Finally plugging $(\ref{dou})$ into the Taylor formula $(\ref{taylor})$, and writing $||\partial_{g^n} \phi|| \leq 2 ||\phi||$, we get that as long as $g^i \in U$ for all $i=0,\ldots,n$,

$$ ||\partial_g \phi|| \leq \frac{||\partial_{g^n} \phi||}{n} + \sum_0^{n-1} ||\partial_{g^i} \partial_g \phi|| \leq  C_V \frac{1}{n} + \frac{vol(V^3)}{vol(V \cap V^{-1})} \frac{1}{n_V(Q)} \frac{1}{L} \leq C_V \frac{1}{n} +\frac{1}{2Kn_V(Q)}$$
since $L > 2K vol(V^3)/vol(V \cap V^{-1})$, and where $C_V=2\frac{vol(V^2)}{vol(V \cap V^{-1})}$. We thus obtain $(\ref{eqq})$ and we are done.\\

This concludes the proof of Lemma \ref{GY1}. The proof of the commutator shrinking property, Item 3) of Lemma \ref{GY2}, is similar and we do not include it here. We also leave out the details of the remainder of the proof of Theorem \ref{gleason-yamabe}, namely the proof that the set of one-parameter subgroups is finite dimensional space on which $G$ acts with central kernel. We refer the reader to Tao's book \cite{tao-hilbert} for these missing items.

\subsection{The Gleason-Yamabe lemmas for approximate groups (Step 2)}

We now return to our main theorem (Theorems \ref{BGT} and \ref{strong}) and to the second step of its proof. Here the intent is to adapt the Gleason-Yamabe lemmas to the setting of approximate groups. In much the same way as these lemmas were used in the locally compact setting to exhibit a small normal subgroup modulo which the group is Lie, the adaptation of the Gleason-Yamabe lemmas to the approximate group setting will yield a finite normal subgroup modulo which the approximate group centralizes a non-trivial subgroup.

\bigskip

Given $g$, let $n_A(g)$ be the largest integer such that $1,g,\ldots,g^n$ belong to $A$.

\bigskip

In order to implement the Gleason-Yamabe lemmas in the approximate group setting, we need an assumption on the approximate group.

\begin{definition}[Strong approximate group]
A $K$-approximate group $A$ is said to be a strong approximate group if
\begin{enumerate}
\item if $1,g,\ldots , g^{1000} \in A^{100}$, then $g \in A$.
\item there is a subset $S=S^{-1} \subset A$ such that $(S^{A^4})^{10^6K^3} \subset A$ and if $1,g,\ldots, g^{10^6K^3} \in A$, then $g \in S$.
\end{enumerate}
\end{definition}

Here $S^{A^4}$ denotes the set of elements of the form $gsg^{-1}$ with $s \in S$ and $g \in A^4$. Observe that, as follows from this definition, if $A$ is a strong $K$-approximate subgroup of $G$, then

\begin{itemize}
\item $\forall g$, $n_A(g) \leq n_{A^{100}}(g) \leq 10^3 n_A(g)$,
\item  $10^6K^3n_S(g) \geq n_A(g)$.
\end{itemize}

 Set

$$||g||_{A}=\frac{1}{n_A(g)}$$
the \emph{escape norm} of $g$ w.r.t $A$. Note that $\|g\|_A=0$ iff the cyclic group $\langle g \rangle$ lies in $A$ entirely.

\begin{lemma}[Gleason-Yamabe lemmas, approximate group version] Given $K\geq 1$, there is $C=C(K)>0$ such that if $A$ is a strong $K$-approximate subgroup of a group $G$, then $\forall g,h \in A$,

\begin{enumerate}
\item $\|h^{-1}gh\|_A \leq C \|g\|_A$
\item $\|gh\|_A \leq C(\|g\|_A + \|h\|_A)$
\item $\|[g,h]\|_A \leq C \|g\|_A \cdot \|h\|_A$
\end{enumerate}

\end{lemma}

We do not repeat the proof, which is essentially identical to the original Gleason-Yamabe lemma that we proved above. The key point here is the second item. For example it implies that if the $\langle g \rangle$ and $\langle h \rangle$ lie both in $A$, then the subgroup generated by $gh$ is also contained in $A$.

Of course in order to be able to apply this lemma, we need to be able to produce a strong approximate subgroup out of the original approximate subgroup $A$. Indeed we prove:

\begin{lemma}[existence of strong approximate subgroups] Given $K>0$, there is $C(K)>0$ such that for every $K$-approximate subgroup $A$ of a group $G$ there is a strong $C(K)$-approximate subgroup $A' \subset A^4$, such that $A$ is contained in a most $C(K)$ left translates of $A'$.
\end{lemma}

This allows to work with $A'$ instead of $A$, and thus apply the Gleason-Yamabe lemmas to it. Unfortunately, we were not able to give a purely combinatorial proof of this lemma. It is a consequence of Hrushovski's local model theorem (Theorem \ref{hrushov}) and of the original Gleason-Yamabe structure theorem for locally compact groups (applied to the model group $G$ from Theorem \ref{hrushov}). Consequently $C(K)$ is ineffective, and there lies the source of ineffectiveness of our main theorem.

\begin{proof} Apply the Gleason-Yamabe structure theorem for locally compact groups to the model group $G$ from Theorem \ref{hrushov}. Up to passing to an open subgroup, we may assume (up to changing $A$ into a commensurate approximate subgroup) that $G$ is a projective limit of Lie groups. But it is easy to see that sufficiently small balls around the identity in exponential coordinates in a connected Lie group (say for a choice of a Euclidean metric on the Lie algebra) satisfy the strong approximate group axioms. Hence we can pull back those sets to $\A^4$ via Theorem \ref{hrushov}, and we are done.
\end{proof}

Note that this argument actually shows that $\{g \in A ; \|g\|_A < \eps\}$ is a positive proportion of $A$, for every $\eps>0$. So up to passing to that subset for $\eps=\frac{1}{C}$, we may assume that $\|g\|_A <\frac{1}{C}$ for all $g \in A$. It is then easy to exploit the approximate group version of the Gleason-Yamabe lemmas and conclude Step 2. Indeed set $H:=\{g \in A; ||g||=0\}$. By $(1)$ and $(2)$, this forms a subgroup of $G$, which is normalized by $A$. Moreover set pick $a \in A$ with $\|a\|_A$ minimal. Then $(3)$ implies that $a$ commutes with every element of $A$ modulo $H$. This concludes Step 2.

For Step 3 and the induction on $\dim L$, we refer the reader to our original paper \cite{bgt}, or to Tao's book \cite{tao-hilbert}.

\section{Applications to diameter bounds and scaling limits of finite groups}

\subsection{Diameter of finite simple groups}

A celebrated conjecture of Babai \cite{babai-seress} proposes the following universal polylogarithmic upper bound on the diameter of all Cayley graphs of finite simple groups:

$$\diam_S(G) \leq C \log(|G|)^D$$
for every finite simple group $G$ with symmetric generating set $S$, where $\diam_S(G):=\inf\{n \in \N ; S^n =G\}$, and absolute constants $C,D>0$.

All attempts to this conjecture so far are based on the classification of finite simple groups, by examining case by case various infinite families, such as the alternating groups (for which the conjecture is still open, see \cite{h-s} for the best available bounds), or finite simple group of Lie type in the $q$ large limit (such as $\PSL_n(q)$ for $n$ fixed). Actually in the latter case, when the rank is bounded, the conjecture has been solved by Helfgott \cite{helfgott-sl2} for $\PSL_2(p)$ and by Pyber-Szabo \cite{pyber-szabo}, and independently Green, Tao and the author \cite{bgt1} in general. In fact it is conjectured that $D=1$ for bounded rank groups (\cite{breuillard-icm}). It has also been extended to some non simple perfect groups such as $\PSL_n(\Z/q\Z)$ with $q$ arbitrary by Varj\'u \cite{varju} (see also \cite{bradford}).

On the other hand, one can very easily derive from our structure theorem for approximate groups, the following weaker, but very general bound.

\begin{theorem}[Diameter bound for finite simple groups \cite{breuillard-tointon}] Given $\eps>0$, there are only finitely many non-abelian finite simple groups admitting a symmetric generating set $S$ with diameter $$\diam_S(G) \geq |G|^\eps.$$
\end{theorem}

\begin{proof}We argue as in the proof of Theorem \ref{gromov}, but instead of Theorem \ref{BGT}, we use the strong form of our structure theorem, i.e. Theorem \ref{strong}. By contradiction, suppose there were infinitely many such groups. Then there would be arbitrarily large integers $n$ and simple groups $G$ with generating sets $S$, such that $|S^{5n}| \leq K|S^n|$, where $K=5^{4/\eps}$, and $n \leq \sqrt{\diam_S(G)}$. For otherwise there would be $n_0 \ge 1$ independent of $G,S$ such that $|G| \geq |S^{5^m n_0}| \geq K^m |S^{n_0}| \geq K^m$ for all $m$ with $5^m n_0 \leq \sqrt{\diam_S(G)}$, but this is not compatible with $\diam_S(G) \geq |G|^\eps$, when $|G|$ is large. By Lemma \ref{approx5}, $S^{2n}$ is a $K^2$-approximate subgroup. We may then apply Theorem \ref{strong} and conclude that $S^{n} \subset XP$, where $|X| \leq C(K)$ and $P$ is a coset nilprogression of size $|P| \leq C(K) |S^{2n}|$. In particular $|S^{2n} \cap \langle P \rangle| \geq |S^{5n}|/KC(K)^2$. By Lemma \ref{schreier}, if $n$ is large enough, $\langle P \rangle$ has index at most $KC(K)^2$ in $G$. But $G$ is a non-abelian simple group, so it has no non-trivial subgroup of bounded index (for every subgroup of index $d$ contains a normal subgroup of index at most $d!$). Hence $G=\langle P \rangle$. However, by definition, $P$ contains a subgroup $H$ normalized by $P$, such that $\langle P \rangle /H$ is nilpotent. Since $G$ is simple, we must have $G=H$, and hence $|G| \leq |XP| \leq C(K)^2 |S^{2n}|$, and hence applying Remark \ref{diambound} $ \diam_S(G) \leq O_K(n)$, which is eventually a contradiction, for $n \leq \sqrt{\diam_S(G)}$.
\end{proof}

The above argument can be strengthened to arbitrary finite groups. For example, in \cite{breuillard-tointon} we prove:

\begin{theorem}[\cite{breuillard-tointon}] Assume that $G$ is a finite group with symmetric generating set $S$ such that $\diam_S(G) \geq |G|^\eps$. Then $G$ has a quotient $Q$ with diameter at least $c_\eps \diam_S(G)$, which contains a cyclic subgroup of index at most $C_\eps$. Here $c_\eps,C_\eps$ are positive constants depending on $\eps$ only.
\end{theorem}

On a Cayley graph, the diameter and the first eigenvalue of the Laplacian enjoy the following general inequality (see \cite{diaconis-saloff} of \cite[Lemma 5.1]{breuillard-tointon}:

$$\lambda_1 \geq \frac{1}{8\diam_S(G)^2}$$

As a consequence of this and the above theorem, we see that every Cayley graph of a finite group $G$ satisfies a spectral gap lower bound of the form

$$\lambda_1 \geq \frac{1}{|G|^\eps}$$
unless it admits a large virtually cyclic quotient with diameter comparable to that of $G$. Such bounds are weaker than what is required for expander graphs, but they can be useful (they can be used for example in the work of Ellenberg-Hall-Kowalski \cite{ehk}). For additional results in this vein, we refer the reader to \cite{breuillard-tointon}.

\subsection{Scaling limits of vertex transitive graphs}

In \cite{bft} Benjamini, Finucane and Tessera considered the problem of determining the scaling limits of finite vertex transitive graphs satisfying the volume bound $(\ref{cond})$ below. It turns that this problem can be analysed thanks to our structure theorem for approximate subgroups. Given constants $C,d>0$, consider the family $\mathcal{F}_{C,d}$ of all finite, vertex transitive connected graphs $X$ such that
\begin{equation}\label{cond}
\frac{|X|}{\deg(X)} \leq C \cdot \diam(X) ^d
\end{equation}
where $|X|$ is the number of vertices of $X$ and $\deg(X)$ is the number of neighbors of a given vertex in $X$ (all vertices have the same number of neighbors, since $X$ is assumed vertex transitive). We denote by $\widehat{X}$ the finite metric space defined on the set $X$ by the graph distance renormalized by a factor $\frac{1}{\diam(X)}$ so that $\widehat{X}$ has diameter $1$.
They prove:

\begin{theorem}[\cite{bft}]\label{toruslimit} For every sequence $\{X_n\}_{n \geq 1}$ of finite graphs belonging to $\mathcal{F}_{C,d}$ for some fixed $C,d$, and such that $\diam(X_n) \to +\infty$, there is a subsequence $\{X_{n_k}\}$ such that $\widehat{X_{n_k}}$ converges in the Gromov-Hausdorff topology towards a torus $\R^q/\Z^q$, endowed with a left-invariant Finsler metric.
\end{theorem}

In other words: scaling limits of vertex transitive graphs with (polynomially) large diameter are flat tori equipped with a norm.  We give a complete proof of Theorem \ref{toruslimit} in this section. In \cite{bft} the stronger condition $|X| \leq C \cdot \diam(X) ^d$ is assumed, but the proof works just as well assuming the weaker bound $(\ref{cond})$. When $\deg(X)$ is uniformly bounded it is possible to further show that $q \leq d$ and that the limit norm is polyhedral, see \cite{bft}, but we will not pursue this here.

\bigskip

A left-invariant Finsler metric on $\R^k/\Z^k$ is nothing but the choice of a norm $\|\cdot \|$ on $\R^k$, which induces a translation invariant distance on the torus $\R^k/\Z^k$  by setting $d([x],[y]) = \inf_{u \in \Z^k} \|x-y+u\|$.

\bigskip

 Here is a graphic way to illustrate the above theorem: the result implies that it is impossible to approximate the Euclidean sphere $S^2$ by finer and finer finite vertex transitive graphs. The soccer ball (or truncated icosahedron, which is vertex transitive) is perhaps (one of) the best attempts at this game, see \cite{gelander, benjamini-book}.

\subsubsection{Gromov compactness criterion}

The Gromov-Hausdorff metric is the natural distance on the space of compact metric spaces up to isometry extending the Hausdorff topology on compact subsets of a given metric space. We refer the reader to the books \cite{gromov-book, bbi} for background on this notion. For the moment we only recall the well-known:

\bigskip

\noindent{{\bf Gromov compactness criterion:}}\emph{ A family $\mathcal{F}$ of compact metric spaces is relatively compact in the GH-topology if and only if it is uniformly bounded in diameter, and for all $\eps>0$ there is $N(\eps) \in \N$ such that every $X \in \mathcal{F}$ can be covered by at most $N(\eps)$ balls of radius $\eps$.}

\bigskip

It easily follows from this criterion that the space of all compact metric spaces up to isometry is complete for the GH-metric. In fact this is the non-trivial part of the statement (it is easy to see that the above criterion is equivalent to the total boundedness of the family $\mathcal{F}$, and clearly a totally bounded subset of a complete metric space is relatively compact). Completeness is only proved for path metric spaces in Gromov's book \cite[Prop 5.2]{gromov-book}, but it is folklore that it holds in general, see e.g. \cite{rouyer}, or \cite[Prop. 1.8.7]{tao-hilbert} for a proof with ultraproducts.

\bigskip

We prove Theorem \ref{toruslimit} in two stages, first showing the existence of a limit and then identifying it.

\bigskip

First note that pulling back the graph metric on $X \in \mathcal{F}_{C,d}$ to the group of automorphisms of the graph, we get a natural word metric on this group. To this effect pick a base point $x_0 \in X$ and consider the subset $S \subset G:=Aut(X)$ of those $s \in S$ which either fix $x_0$ or send it to a neighbor. Since $X$ is connected and vertex transitive $G$ is generated by $S$ and $S=S^{-1}$. Moreover the preimage under the orbit map $G \to X$, $g \mapsto g\cdot x_0$ of the ball $B(x_0,n)$ of radius $n$ around $x_0$ is precisely $S^n$. In particular $\diam(X)=\diam_S(G)$ and $$\frac{|S^n|}{|S|} = \frac{|B(x_0,n)|}{\deg(X)}$$ for all $n \geq 1$. This allows to easily translate the statements from $X$ to $G$.

\begin{proof}[Proof of the existence of a limit in Theorem \ref{toruslimit}]  To apply the Gromov compactness criterion above to the family $\mathcal{F}=\{\widehat{X_n}\}_{n\geq 1}$, it is enough to check that given $\eps>0$, there is $K_\eps >0$ such that every ball of radius $r \geq \eps$ can be covered by at most $K_\eps$ balls of radius $\frac{r}{2}$.  In fact we will show the stronger property (needed only for the identification of the limit) that $K_\eps$ can be taken to be independent of $\eps$, if $n$ is large enough. This follows easily from Theorem \ref{bt-thm}. Indeed setting $K=9^d$ in Theorem \ref{bt-thm} we get an integer $n_0=n_0(K) \geq 1$. And we see that if $\diam(X_n)$ is large enough, there must be some $n_1 \in [n_0,\eps \diam(X_n)]$ such that $|S^{3n_1}| \leq K |S^{n_1}|$, for otherwise $|S^{3^{m} n_0}| \geq 9^{dm} |S^{n_0}|$ for all $m$ with $3^{m} n_0  \leq \eps \diam(X_n)$, and thus $|G| \geq 9^{dm}|S|$ for the maximal such $m$, which implies $C \diam(X_n)^d \geq 9^{dm} \geq (\eps \diam(X_n) / 3n_0)^{2d}$, a contradiction for $n$ large. So we apply Theorem \ref{bt-thm} and get that $|B(x_0,4k)| \leq K' |B(x_0,k)|$ for all $k \geq \eps \diam(X_n)$ and all $n$ large enough, for some $K'>0$ independent of $n$. From this we conclude that every ball of radius $4k$ can be covered by at most $K'$ balls of radius $2k$, and this ends the proof.
\end{proof}

We remember from this argument, that there exists some $K=K(C,d)>0$ such that if $X$ is a limit of a sequence of $\widehat{X_n}$ with $X_n$ satisfying $(\ref{cond})$ and with $\diam(X_n) \to +\infty$, then $X$ is $K$-doubling in the sense that every ball of radius $r>0$ can be covered by at most $K$ balls of radius $\frac{r}{2}$. It follows immediately from this that $X$ has finite Hausdorff dimension, hence \emph{finite topological dimension}.

\bigskip

\subsubsection{Identification of the limit: outline} We now turn to the identification of the limit. This part is almost purely formal in that it follows from a simple combination of known facts (we differ here from the original treatment of \cite{bft}). Here is the argument. If $\widehat{X_n}$ converges to some compact metric space $X$, then $G:=Isom(X)$ must act transitively on $X$, because $G_n:=Isom(X_n)$ acts transitively on $X_n$. Moreover, $X$ is pathwise connected and locally connected (because the balls in the graphs $X_n$ are connected), and has finite topological dimension. From this, a lemma of Montgomery-Zippin, which is a simple consequence of the Peter-Weyl theorem, tells us that $G$ is a (compact) Lie group. Moreover up to passing to a subsequence (equivariant GH-convergence), we may assume that there is an $\eps_n$-representation of $G_n$ in $G$, which is $\eps_n$-dense, for some $\eps_n \to 0$. A well-known result (the so-called \emph{Ulam stability} of almost representations of finite groups in compact groups, due to Alan Turing in this setting and further generalized by Kazhdan and others) ensures that the $\eps_n$-representation must be close to a genuine representation. Then Jordan's lemma on finite subgroups of Lie groups and the almost density of the representations imply that $G^0$ must be abelian, hence $X=G^0$ is a torus. The metric on $X=G^0$ is geodesic and left-invariant: all such metrics on a torus are Finsler (i.e. given by a norm).

\bigskip

\subsubsection{Montgomery-Zippin lemma} We now pass to the details. Set $G:=Isom(X)$. We already know that $X$ is pathwise connected, locally connected, $G$-transitive, and finite dimensional. We have the following well-known lemma:

\begin{lemma}\label{MZlemma} If $G$ is a compact group acting faithfully and transitively on a finite-dimensional, connected and locally connected topological space, then $G$ is a Lie group with finitely many connected components.
\end{lemma}

\begin{proof} The lemma holds for locally compact groups in general and is due to Montgomery-Zippin (see \cite[\S 6.3, p. 243]{montgomery-zippin}). It easily follows from the Gleason-Yamabe theorem. In the case of compact groups, only Peter-Weyl is needed: we include the proof for the reader's convenience.

According to the Peter-Weyl theorem, there is a nested sequence of closed normal subgroups $N_i$ in $G$ such that $G/N_i$ is a Lie group and $\bigcap_i N_i = \{1\}$. We have to show that $N_i$ is trivial for $i$ large. Let $X_i:=X/N_i$ be the quotient space under the action of $N_i$. The space $X$ is naturally a projective limit of the $X_i$'s, i.e. $X=\varprojlim X_i$. Note that the $X_i$'s are manifolds: they are homogeneous spaces of the compact Lie group $G/N_i$. Since $X$ is finite dimensional as a topological space, we must have $\dim(X_i)=\dim(X_{i_0})$ when $i$ is larger than some $i_0$. This implies that $N_{i_0}/N_{i}$ is zero-dimensional if $i \geq i_0$. Being a compact Lie group, it must be finite. Therefore $X_{i}$ is a finite normal cover of $X_{i_0}$ with finite group of deck transformations $N_{i_0}/N_{i}$. The group $N_{i_0}$ is a  projective limit of finite groups, i.e. it is profinite.

We claim that $N_{i_0}$ is finite. This follows from the local connectedness of $X$. Indeed, by covering theory, the universal cover of $X_{i_0}$ is a cover of each $X_i$, $i \geq i_0$. Therefore, given $x \in X_{i_0}$ there exists an open neighborhood $U$ of $x$ such that $\pi_i^{-1}(U)$ is a disjoint union of open subsets of $X_i$ each homeomorphic to $U$ and permuted by the action of $N_{i_0}/N_i$ for $i \geq i_0$. Hence $\Omega:=\varprojlim \pi_i^{-1}(U)$ is homeomorphic to $U \times N_{i_0}$. Note that since $X=\varprojlim X_i$, $\Omega$ is also an open subset of $X$, which is not locally connected at any point, unless $N_{i_0}$ is finite. This ends the proof.
\end{proof}

Note in passing that the hypothesis that $X$ be locally connected is important to avoid examples such as the solenoid $(\R \times \Q_2) / \Z[\frac{1}{2}]$, which is a connected and finite dimensional compact group, but not a Lie group.

Note also that, since $X$ is connected, the connected component of the identity $G^0$ in $G$ acts transitively on $X$.

\bigskip

\subsubsection{Ulam stability and almost representations}

We now recall the notion of almost representation. Given a target group $G$, endowed with a left-invariant distance $d$ making it a metric space, a mapping $\phi$ from a group $\Gamma$ to $G$, is called \emph{an $\eps$-representation} if for every $x,y \in \Gamma$,
$$d(\phi(xy),\phi(x)\phi(y)) \leq \eps.$$

We say moreover that $\phi$ is \emph{$\eps$-dense} if for each $g \in G$ there is $x \in \Gamma$ such that $d(g, \phi(x)) \leq \eps$.

\bigskip

 Ulam stability (see \cite{ulam}) is the general problem of determining to what extent an $\eps$-representation of $\Gamma$ is $\delta(\eps)$-close to a genuine group homomorphism, i.e. does there exists a group homomorphism $\rho: \Gamma \to G$ such that $d(\phi(x), \rho(x)) \leq \delta(\eps)$, for some $\delta(\eps)$ tending to $0$ as $\eps \to 0$. A little known and recently rediscovered \cite{bft,gelander} 1938 paper of Alan Turing \cite{turing} deals precisely with this problem in the case when $\Gamma$ is a finite group and $G$ a connected Lie group (endowed with a left-invariant Riemannian metric):

\begin{theorem}[Turing's theorem] Suppose $G$ is a connected compact Lie group, then there is $C>0$ such that every $\eps$-dense $\eps$-representation of some finite group into $G$ is $C\eps$-close to a group homomorphism. In particular, if there is a sequence of finite groups  $\Gamma_n$  and  $\eps_n$-representations $\phi_n: \Gamma_n \to G$ that are $\eps_n$-dense in $G$ for some $\eps_n \to 0$, then $G$ is a compact abelian Lie group.
\end{theorem}

The second part of the statement follows immediately from the combination the first and of Jordan's theorem on finite subgroups of Lie groups: that they contain an abelian subgroup of bounded index (\cite{jordan,curtis-reiner,breuillard-jordan}). The first part actually holds without the assumption of $\eps$-density (we do not need this refinement for the proof of Theorem \ref{toruslimit}), as follows from the following more general result of Kazhdan \cite[Theorem 1]{kazhdan}. Let $G=\mathcal{U}(\mathcal{H})$ be the unitary group of a separable Hilbert space, endowed with the bi-invariant distance $d(g,h):=\|g-h\|$.

\begin{theorem}[Ulam stability of amenable groups]\label{kaz} Every $\eps$-representation $\phi$ of an amenable group $\Gamma$ into $G$ is $\eps$-close to a genuine representation $\rho$, i.e. $d(\phi(\gamma),\rho(\gamma)) \leq \eps$ for all $\gamma \in \Gamma$.
\end{theorem}

Note that every finite group is amenable, so this applies uniformly to all finite groups $\Gamma$, and also note that every compact Lie group is a subgroup\footnote{Ulam stability may fail for certain target groups other than $\mathcal{U}(\mathcal{H})$, for example it fails for $G$ the $p$-adic integers, see \cite[Prop 1.]{kazhdan}} of the  unitary matrices $\mathcal{U}(\C^n)$ for some $n$ (the above distance being bi-Lipschitz to any left-invariant Riemannian metric), so the result applies to this situation as well and implies Turing's theorem.

For a very short proof of Kazhdan's theorem, as well as a modern view on Ulam stability, we refer the reader to the paper \cite{bot}, in particular Theorem 3.2 there.

\subsubsection{Equivariant GH convergence} When each term of a sequence of compact metric spaces comes equipped with a group action by isometries, one can strengthen the Gromov compactness theorem and argue that under the same hypotheses, one can find a subsequence which converges in an almost equivariant way.  Namely we have:

\bigskip

\noindent{{\bf Equivariant Gromov-Hausdorff compactness:}} \emph{Let $(X_n,d_n)\to (X,d)$ be a converging sequence of compact metric spaces. Let $\Gamma_n:=Isom(X_n)$  Then there is a closed subgroup of isometries $G$ of $X$ and subsequence $\{X_{n_k}\}$, such that $\{(X_{n_k},d_{n_k})\}$ equivariantly GH converges to $(X,G)$, in other words, there are maps $f_k: X_{n_k} \to X$ and $\phi_k: \Gamma_{n_k} \to G$ with $\eps_k$-dense image in $X$ and $G$ respectively (where $\eps_k \to 0)$, such that
\begin{equation}\label{conve}
 \sup_{x \in X_{n_k}, \gamma \in \Gamma_{n_k}} d(f_k(\gamma x) , \phi_k(\gamma) f_k(x)) \xrightarrow[k \to +\infty]{}0
\end{equation}
}
\bigskip

This is a classical fact, which holds more generally for locally compact metric spaces and pointed GH convergence, see for example \cite[Prop. 3.6]{fukaya-yamaguchi}. We give a proof for convenience.

\begin{proof} By definition there are almost isometries $f_n : X_n \to X$ and $g_n:X \to X_n$ such that $f_n \circ g_n$ and $g_n \circ f_n$ are both uniformly close to the identity. Pick a nested sequence of $\eps_k$-dense finite subsets $F_k$ in $X$, where $\eps_k \to 0$. The restrictions of elements in $f_n \Gamma_n g_n$ to $F_k$ form a compact set of maps from $F_k$ to $X$. By the usual diagonal process we can extract a subsequence of $(X_n)_n$ such that these compact sets converge as $n$ tends to infinity for every given $k$. Let $G_k$ be the limit. The restriction map $G_{k+1} \to G_k$ is surjective. The projective limit $G:=\varprojlim G_k$ is a compact set of maps on $\cup_k F_k$. Maps in $G_k$ are distance preserving, hence so are maps in $G$. Thus maps in $G$ extend by continuity to isometries of $X$.

Given $k$ there is $n_k$ such that $G_k$ is $\eps_k$-close to the restriction of $f_n\Gamma_ng_n$ to $F_k$ for $n=n_k$. Thus we may find for each $\gamma \in \Gamma_{n_k}$ an element $\phi_k(\gamma)$ in $G_k$ (hence also one in $G$) such that $f_{n_k}\gamma g_{n_k}$ is $3\eps_k$-close to $\phi_k(\gamma)$ on $X$. Similarly the image of $\phi_k$ is $3\eps_k$-dense in $G$. It is then straightforward to check $(\ref{conve})$ and that $G$ is indeed a subgroup.
\end{proof}

Note that it follows directly from $(\ref{conve})$ that $\phi_{n_k}$ is a $\delta_k$-representation of $\Gamma_{n_k}$ inside $G$ for some $\delta_k \to 0$.

\bigskip

By Lemma \ref{MZlemma}, in the setting of Theorem \ref{toruslimit}, the limit group $G$ is a compact Lie group and the groups $\Gamma_n$ are finite. Therefore Turing's theorem (or the combination of Kazhdan and Jordan) implies that $G^0$ is an abelian compact Lie group, i.e. a torus $\R^k/\Z^k$. We have $X=G^0/H$ for some closed subgroup $H$. But since $G^0$ acts faithfully on $X$ (it is a subgroup of isometries) and $G^0$ is abelian, we must have $H=1$, i.e. $X=G^0$.

\subsubsection{Invariant length metrics on Lie groups} To conclude the proof of Theorem \ref{toruslimit} it remains to identify the metric on the limit space $X$. We know that it is a \emph{length metric} a.k.a intrinsic metric, i.e. the distance between two points coincides with the infimum of the length of paths between two points. This is so, because $X$ is a GH limit of graphs and the graph metric is a length metric. Since $X=\R^k/\Z^k$, the distance on $X$ lifts naturally to a length metric on $\R^k$, which is invariant under translations. So the last part of Theorem \ref{toruslimit} is now a consequence of the following simple and well-known lemma, of which we give a proof for convenience.

\begin{lemma}If $d$ is a translation invariant length metric on $\R^k$, then there is a norm $\|\cdot \|$ on $\R^k$ such that $d(x,y)=\|x-y\|$ for all $x,y$.
\end{lemma}

\begin{proof}Set $\|x\|:=\limsup_{t \to 0} \frac{1}{t}d(0,tx)$. Clearly, by translation invariance and the triangle inequality $d(0,x) \leq n d(0,\frac{x}{n})$ for every $n \in \N$. It follows that $d(0,x) \leq \|x\|$. It is also straightforward to verify that $\|x+y\| \leq \|x\| + \|y\|$ and that $\|\lambda x\|=\lambda \|x\|$ if $\lambda >0$. Hence $\|x\|$ is a norm on $\R^k$. Moreover we also have $\|x\|:=\liminf_{t \to 0} \frac{1}{t}d(0,tx)$. Indeed if not there would be some $\beta <1$ and arbitrarily small $s,t$ such that $\frac{1}{s} d(0,sx) \leq  \beta \frac{1}{t}d(0,tx)$. If $s$ is chosen much smaller than $t$, say $t=ns + s'$, $0 \leq s' \leq s$, $n\geq 1$, then by the triangle inequality and translation invariance $d(0,tx) \leq nd(0,sx) + d(0,s'x)$, which eventually contradicts the previous inequality.

So $\|x\|:=\lim_{t \to 0} \frac{1}{t}d(0,tx)$. By local compactness of $\R^k$, this limit is uniform over $x$ with $\|x\|=1$. From this it is easy to conclude that $\|x\|=d(0,x)$. Indeed if not, then $d(0,x)\leq (1-\eps)\|x\|$ for some $x$ and $\eps>0$. Moreover given any $\delta>0$, there are points $x_0=0, \ldots, x_n=x$ such that $d(x_i,x_{i+1}) < \delta$ and $d(x_0,x_1)+\ldots+d(x_{n-1},x_n) \leq (1-\frac{\eps}{2})\|x\|$. However $d(x_i,x_{i+1})$ is uniformly very close to $\|x_i - x_{i+1}\|$, but $\|x\| \leq \|x_0-x_1\|+\ldots+\|x_{n-1}-x_n\|$ by the triangle inequality. A contradiction.
\end{proof}

This lemma is a special case of the following much more general result of Berestovski, which gives a complete characterization of left invariant length metrics on Lie groups.

\bigskip

Before we state Berestovski's theorem, let us give a general construction of left invariant length metrics on a Lie group. Pick a norm on the Lie algebra $\g$ of the Lie group $G$ and use it to define the length of piecewise differential paths in exactly the same way as left invariant Riemannian metrics on $G$ are defined, except that the norm is now arbitrary instead of Euclidean. This obviously gives rise to a left invariant (Finsler) length metric on $G$. More generally, we may consider a norm on a vector subspace $V$ of $\g$, which generates $\g$ as a Lie algebra. Similarly this allows to define the length of piecewise differentiable \emph{horizontal} paths, that is paths whose derivative, once translated back to the origin, everywhere belongs to $V$. The property of $V$ that it generates $\g$ can be shown (this is the Chow-Rachevski theorem from subriemannian geometry) to imply that every two points in $G$ can be joined by a horizontal piecewise differentiable continuous path. Hence this yields a well-defined length metric on $G$. Metrics arising from the above construction are called left invariant Carnot-Caratheodory-Finsler metrics.

\bigskip

When $G$ is abelian, then every subspace of $\g$ is a subalgebra, so this refinement does not produce new examples of left invariant metrics. But if $G$ is nilpotent (e.g. the Heisenberg group), then many new length metrics arise this way. In fact it can be shown (\cite{bere}) that $G$ abelian and $G=\R \ltimes \R^n$, where $\R$ acts by homotheties, are the only Lie groups in which all Carnot-Caratheodory-Finsler metrics are Finsler. We can now state:

\begin{theorem}[Berestovski \cite{bere}] Every left invariant length metric on a connected Lie group is Carnot-Caratheodory-Finsler.
\end{theorem}

This ends our discussion of Theorem \ref{toruslimit}.

\bigskip

\noindent {\bf Acknowledgement.} These lectures were prepared for the workshop on \emph{Additive and Analytic Combinatorics} held at the IMA in Minneapolis from September 29 to October 3, 2014. Due to a health problem, I was unfortunately not able to give the lectures. I am very grateful to Terry Tao, who replaced me on the spot. I also thank T. Tao, R. Tessera and M. Tointon for their valuable comments on the text. The author is partially supported by ERC grant no. 617129.

\end{document}